\documentclass[10pt,reqno]{amsart}

\usepackage[all]{xy}
\usepackage{tom}
\usepackage{epic,eepic}

\newcommand{\ux}{{\boldsymbol x}}
\newcommand{\uz}{{\boldsymbol z}}
\newcommand{\Kx}{\K[[\ux]]}
\newcommand{\qh}{QH}
\newcommand{\sqh}{SQH}
\newcommand{\rsqh}{rSQH}
\newcommand{\csqh}{cSQH}
\newcommand{\ph}{PH}

\newcommand{\rsph}{rSPH}
\newcommand{\csph}{cSPH}
\newcommand{\Aa}{A}
\newcommand{\AAa}{AA}
\newcommand{\AC}{AC}
\newcommand{\AAC}{AAC}

\newcommand{\SND}{IND}

\newcommand{\SNND}{INND}

\DeclareMathOperator{\In}{in}
\DeclareMathOperator{\gr}{gr}
\DeclareMathOperator{\jj}{j}
\DeclareMathOperator{\tj}{tj}

\newcommand{\tom}[1]{}
\newcommand{\lang}[1]{}
\newcommand{\kurz}[1]{#1}
\newcommand{\bmath}{\kurz{\begin{math}}\lang{\begin{displaymath}} }
\newcommand{\emath}{\kurz{\end{math}}\lang{\end{displaymath}} }

\kurz{\renewcommand{\K}{K}
  \newcommand{\LL}{L}
  \newcommand{\ideal}{\subset}
  \newcommand{\unideal}{\subseteq}}
\lang{\renewcommand{\K}{{\mathds K}}
  \newcommand{\LL}{{\mathds L}}
  \newcommand{\ideal}{\lhd}
  \newcommand{\unideal}{\unlhd}
  }

\catcode`@=11
\renewenvironment{enumerate}{%
  \listparindent0pt 
  \ifnum \@enumdepth >3 \@toodeep\else
      \advance\@enumdepth \@ne
      \edef\@enumctr{enum\romannumeral\the\@enumdepth}\list
      {\csname label\@enumctr\endcsname}{\usecounter
        {\@enumctr}\def\makelabel##1{\hss\llap{\upshape##1}}}\fi
      \itemsep1ex\partopsep1ex\labelsep1ex
}{%
  \endlist
}
\catcode`@=\active%

\begin{document}

   \parindent0cm

   \title[Normal forms]{Normal Forms of Hypersurface Singularities in
     Positive Characteristic}
   \dedicatory{Dedicated to Sabir Medzhidovich Gusein-Zade at the
     occasion of his 60th birthday}
   \author{Yousra Boubakri}
   \address{Universit\"at Kaiserslautern\\
     Fachbereich Mathematik\\
     Erwin--Schr\"odinger--Stra\ss e\\
     D --- 67663 Kaiserslautern
     }
   \email{yousra@mathematik.uni-kl.de}
   \author{Gert-Martin Greuel}
   \address{Universit\"at Kaiserslautern\\
     Fachbereich Mathematik\\
     Erwin--Schr\"odinger--Stra\ss e\\
     D --- 67663 Kaiserslautern\\
     Tel. +496312052850\\
     Fax +496312054795
     }
   \email{greuel@mathematik.uni-kl.de}
   \urladdr{http://www.mathematik.uni-kl.de/\textasciitilde greuel}
   \author{Thomas Markwig}
   \address{Universit\"at Kaiserslautern\\
     Fachbereich Mathematik\\
     Erwin--Schr\"odinger--Stra\ss e\\
     D --- 67663 Kaiserslautern\\
     Tel. +496312052732\\
     Fax +496312054795
     }
   \email{keilen@mathematik.uni-kl.de}
   \urladdr{http://www.mathematik.uni-kl.de/\textasciitilde keilen}



   \date{\today}

     
   \begin{abstract}
     The main purpose of this article is to lay the foundations for a
     classification of isolated hypersurface singularities in positive
     characteristic. Although our article is in the spirit of Arnol'd
     who classified real an complex hypersurfaces in the 1970's with
     respect to right equivalence, several new phenomena occur in
     positive characteristic. Already the notion of isolated
     singularity is different for right resp.\ contact equivalence
     over fields of characteristic other than zero. The heart of this
     paper consists of the study of different notions of
     non-degeneracy and the associated piecewise filtrations induced
     by the Newton diagram of a power series $f$. We introduce the
     conditions \AC\ and \AAC\ which modify and generalise the
     conditions \Aa\ and \AAa\ of Arnol'd resp.\  Wall and which allow the
     classification with respect to contact equivalence in any
     characteristic. Using this we deduce normal forms and rather
     sharp determinacy bounds for $f$ with respect to right and
     contact equivalence. We apply this to
     classify hypersurface singularities of low modality in positive
     characteristic. 
   \end{abstract}

   \maketitle

   \section{Introduction}

   \lang{
     Throughout this paper $\K$ shall be an algebraically closed field of
     arbitrary characteristic unless explicitly stated otherwise. By
     \begin{displaymath}
       \Kx=\K[[x_1,\ldots,x_n]]= 
       \left\{\sum_{\alpha\in\N^n}a_\alpha\cdot \ux^\alpha\;\Big|\;a_\alpha\in\K\right\}
     \end{displaymath}
     we denote the formal power series ring over $\K$ in $n\geq 2$
     indeterminates $x_1,\ldots,x_n$ using the usual multiindex notation
     $\ux^\alpha=x_1^{\alpha_1}\cdots x_n^{\alpha_n}$ for
     $\alpha=(\alpha_1,\ldots,\alpha_n)\in\N^n$. Moreover, we denote by
     \begin{displaymath}
       \m=\langle x_1,\ldots,x_n\rangle\ideal \Kx
     \end{displaymath}
     the unique maximal ideal of $\Kx$, so that the set of units in
     $\Kx$ is $\Kx^*=\Kx\setminus\m$.
   }
   \kurz{
     Throughout this paper $\K$ denotes an algebraically closed field of
     arbitrary characteristic unless explicitly stated otherwise. By
     $\Kx=\K[[x_1,\ldots,x_n]]$, $n\geq 1$,
     we denote the formal power series ring over $\K$, and by
     $\m=\langle x_1,\ldots,x_n\rangle$ 
     the maximal ideal of $\Kx$.
   }

   In
   the classification of power series $f\in\Kx$ there are two natural
   equivalence relations, right equivalence and contact equivalence.
   Two power series $f,g\in\Kx$ are
   \emph{right equivalent} \lang{if and only} if there is an
   automorphism $\varphi\in\Aut(\Kx)$ such that $f=\varphi(g)$, and we
   denote this by $f\sim_r g$. 
   \lang{
     If we replaced $\K$ by the complex
     numbers and formal power series by convergent ones then $\varphi$
     would induce an isomorphism of the zero fiber of $f$ as well as of
     close by fibers. That is how we should interpret 
     right equivalence also in this more general setting.

   If we are only interested in the geometry of the zero fiber, then
   the second equivalence relation is the appropriate one. } 
   We call $f,g\in\Kx$ 
   \emph{contact equivalent} \lang{if and only} if there is an
   automorphism $\varphi\in\Aut(\Kx)$ and a unit $u\in\Kx^*$ such that
   $f=u\cdot \varphi(g)$, and we
   denote this by $f\sim_c g$. 
   \lang{The idea here is, that $\varphi$ and
   $u$ still induce an isomorphism of the zero fibers of $f$ and $g$. 

   However, we have to replace the geometric notion of the zero fiber
   by the algebraic counterpart of its coordinate ring. That is, for a
   power series $f\in\Kx$ we call $R_f=\Kx/\langle f\rangle$ the
   induced \emph{hypersurface singularity}. 
   We obviously have
     \begin{displaymath}
       f\sim_c g\;\;\;\Longleftrightarrow\;\;\; R_f\cong R_g,
     \end{displaymath}
     i.e.\ $f$ and $g$ are contact equivalent if and only if the
     induced hypersurface singularities are isomorphic as local analytic
     $\K$-algebras.}
   \kurz{Note that $f,g\in\Kx$ are contact equivalent
     if and only if their associated \emph{hypersurface singularities}
     $R_f=\Kx/\langle f\rangle$ and $R_g=\Kx/\langle g\rangle$ are
     isomorphic as local analytic $\K$-algebras.}

   \lang{
   Over the complex numbers we would say that the origin is an
   isolated singular point of $f$ if $f$ is not singular at any point
   close-by, i.e.\ the origin is the only common zero of the partial
     derivatives of $f$. We have to reformulate this algebraically so 
   that it works over any field.} 
   For a power series $f\in\Kx$ we
   denote by
   \bmath
     \jj(f):=\langle f_{x_1},\ldots,f_{x_n}\rangle \unideal  \Kx
   \emath
   the \emph{Jacobian ideal} of $f$, \lang{i.e.\ the ideal} generated by the
   partial derivatives of $f$, and we call the associated algebra
   \bmath
     M_f:=\Kx/\jj(f)
   \emath
   the \emph{Milnor algebra} of $f$ and its dimension
   \bmath
     \mu(f):=\dim_\K(M_f)
   \emath
   the \emph{Milnor number} of $f$. We then call \lang{the origin an
   isolated singular point of $f$, or we simply call } $f$ an
   \emph{isolated singularity} if $\mu(f)<\infty$, which is 
   equivalent to the existence of a non-negative integer $k$ such that
   $\m^k\subseteq\jj(f)$. 

   \lang{Similarly, over $\C$ we would call the origin an isolated singular
   point of the hypersurface singularity defined by $f$ if this
   hypersurface singularity has no other singular point close-by,
   i.e.\ the origin is the only common zero of $f$ and its partial
   derivatives. Algebraically we thus consider}\kurz{Similarly we define } 
   the \emph{Tjurina ideal}
   \bmath
     \tj(f):=\langle f,f_{x_1},\ldots,f_{x_n}\rangle=\langle f\rangle+\jj(f)\unideal\Kx
   \emath
   of $f$, the associated \emph{Tjurina algebra}
   \bmath
     T_f:=\Kx/\tj(f)
   \emath
   of $f$ and its dimension
   \bmath
     \tau(f):=\dim_\K(T_f),
   \emath
   the \emph{Tjurina number} of $f$. We then call \lang{the origin an
   isolated singular point of the hypersurface singularity $R_f$, or
   we will simply call } $R_f$ an \emph{isolated hypersurface singularity} if
   $\tau(f)<\infty$, or equivalently $\m^k\subseteq \tj(f)$ for some
   $k\geq 0$. 

   It is straight forward to see that the Milnor number is \emph{invariant} under right equivalence and
   the Tjurina number is \emph{invariant} under contact equivalence. 
    
   Our principle interest is the classification of  power series in
   positive characteristic with
   respect to right respectively 
   contact equivalence, where the latter is the same as to say that we
   are interested in classifying hypersurface singularities up to
   isomorphism. In order to have good finiteness conditions at hand we
   restrict to the case that $f$ is an
   isolated singularity for right equivalence respectively 
   that $R_f$ is an isolated hypersurface singularity  for
   contact equivalence. Note that these are two distinct conditions
   in positive characteristic  (see also \cite{BGM10}).

   A first important step in the attempt to classify singularities
   from a theoretical point of view as well as from a practical one is
   to know that the equivalence class is determined by a finite number
   of terms of the power series $f$ and to find the corresponding degree
   bound.  We say that $f$ is \emph{right} respectively
   \emph{contact $k$-determined} if $f$ is
   right respectively contact equivalent to every $g$ which coincides
   with $f$ up to order $k$. 

   In \cite{BGM10} we have shown that $f$ is
   finitely right repectively contact determined if and only if $\mu(f)$
   respectively $\tau(f)$ is finite, and we have shown that
   $2\cdot\mu(f)-\ord(f)+2$ respectively $2\cdot\tau(f)-\ord(f)+2$ is
   an upper bound for the determinacy. 
   Here $\ord(f)$ denotes the order.
   In Corollaries
   \ref{cor:fdre} and \ref{cor:fdce} we show how this degree bound can
   be considerably improved when the singularities satisfy the
   conditions \AAa\ 
   resprectively \AAC\  introduced in Section \ref{sec:normalforms}
   (see also the examples in Section \ref{sec:examples}). 

   Once we know that a finite number of terms of $f$ suffices to
   determine its equivalence class then we would like to determine a
   \emph{normal form} for $f$, i.e.\ an ``efficient'' representative for
   the equivalence class. This is in general a difficult task.
   The first classes of singularitites one comes across are those
   which have a \emph{quasihomogeneous} representative. In
   characteristic zero they are determined by the fact that the Milnor
   number and the Tjurina number coincide (see Theorem
   \ref{thm:saito}). The next more
   complicated classes of singularities are those which have a
   representative with a quasihomogeneous principal part (that 
   governs its topology over the complex numbers), i.e.\ the
   \emph{right semi-quasihomogeneous} \rsqh\  respectively 
   \emph{contact semi-quasihomogeneous} \csqh\  singularities. These are
   considered in Section \ref{sec:sqh}, and among others we show that
   they are indeed isolated (see Proposition \ref{prop:sqh}).
  
   When obtaining normal forms of  power series which are not
   right semi-quasi\-homo\-geneous the only
   known classification method is due to Arnol'd, introduced in
   \cite{Arn75} (see also \cite[Sec.~12.7]{AGV85}) over the
   complex numbers and slightly generalised by Wall in \cite{Wal99}.
   The method generalises semi-quasihomogeneity and requires the principal part
   $\In_P(f)$ of the power series (with 
   respect to some $C$-polytope $P$) to be an \emph{isolated} singularity and
   its Milnor algebra to have a \emph{finite regular basis} -- see
   Section \ref{sec:piecewisefiltration} and \ref{sec:normalforms}\lang{ for
   the notions}. At the heart lies the study of \emph{piecewise
     filtrations} as introduced by Arnol'd \cite{Arn75} and used by
   Kouchnirenko to study non-degeneracy conditions
   \cite{Kou76}. Section \ref{sec:piecewisefiltration} is devoted to
   these.  
   Arnol'd introduced the important condition \Aa\ and Wall noticed
   that the weaker condition \AAa\ is sufficient to get normal forms.
   In Section
   \ref{sec:normalforms} we generalise
   these conditions \lang{both in the strict form \Aa\  of Arnol'd  and in the weak form
   \AAa\  of Wall}  to the
   situation of contact equivalence, calling them \AC\  and \AAC\
   respectively,  and derive normal forms for right
   as well as for contact equivalence in arbitrary characteristic.
   See Theorem~\ref{thm:nfre} and 
   \ref{thm:nfce} and Corollaries~\ref{cor:nfre} and \ref{cor:nfce}.

   The results on normal forms and degree bounds apply to large
   classes of examples. In Corollary \ref{cor:sqh-aa} we show that
   all \rsqh\  singularities satisfy \AAa\  and all
   \csqh\  singularities satisfy \AAC.
   Moreover, we show that all
   \emph{inner Newton non-degenerate} 
   singularities (for the definition see Remark \ref{rem:snnd}) satisfy both \AAa\  and \AAC\  (see Theorem
   \ref{thm:snnd-aa}) which generalises a result of Wall. In Section \ref{sec:examples} we then
   use our results to derive normal forms for singularities of type
   $T_{pq}$, $Q_{10}$, $W_{1,1}$ and $E_7$ in Arnol'd's notation in positive
   characteristic.

   \section{Quasi- and semi-quasihomogeneous singularities}\label{sec:sqh}

   A polynomial $f\in\K[\ux]$, $\ux=(x_1,\ldots,x_n)$, $n\geq 1$, 
   is called \emph{quasihomogeneous} with
   respect to the weight vector $w\in\Z_{>0}^n$ if all monomials 
   $\ux^\alpha$ of $f$ have the same \emph{weighted degree}
   $d:=\deg_w(\ux^\alpha)=w\cdot \alpha=\sum_{i=1}^nw_i\cdot
   \alpha_i$. We say for short that $f$ is \qh\  of type $(w;d)$. 
   By the \emph{Euler formula} a quasihomogeneous polynomial $f$
   of weighted degree $\deg_w(f):=d$ satisfies
   \begin{displaymath}
     d\cdot f=w_1\cdot x_1\cdot f_{x_1}+\ldots+w_n\cdot x_n\cdot f_{x_n},
   \end{displaymath}
   so that
   \begin{displaymath}
     \jj(f)=\tj(f)\;\;\;\mbox{ if }\;\Char(\K)\not|\;d.
   \end{displaymath}
   In particular, for a quasihomogeneous polynomial $f$ in
   characteristic zero the Milnor number and the Tjurina number
   coincide. A famous result of K.\ Saito states that
   over the complex numbers the converse is true as well
   (up to equivalence) (see \cite{Sai71}). One can check that his
   proof generalises to any algebraically 
   closed field of characteristic zero, a statement which we could not
   find in the literature. 

   \begin{theorem}[Saito, \cite{Sai71}]\label{thm:saito}
     Let $\K$ be an algebraically closed field of characteristic zero,
     and suppose that $f\in\Kx$ is an isolated singularity. Then
     the following are equivalent:
     \begin{enumerate}
     \item $f$ is right equivalent to a quasihomogeneous polynomial.
     \item $f$ is contact equivalent to a quasihomogeneous polynomial.
     \item $\mu(f)=\tau(f)$.
     \item $f\in\jj(f)$.
     \end{enumerate}
   \end{theorem}

   The Milnor and the Tjurina number are important invariants which
   even characterise the singularities for values $0$ and $1$ in any
   characteristic. In fact, let $\ord(f):=\min\{k\;|\;f\in\m^k\}$,
   then,  by the implicit function theorem we see easily for $f\in\m$
   \begin{displaymath}
     \mu(f)=0\;\Longleftrightarrow\;
     \tau(f)=0\;\Longleftrightarrow\;
     \ord(f)=1\;\Longleftrightarrow\;
     f\sim_r x_1\;\Longleftrightarrow\;
     f\sim_c x_1.
   \end{displaymath}
   If $\ord(f)\geq 3$ we have $\mu(f)\geq\tau(f)\geq n+1\geq 2$. If
   $\mu(f)=1$, then $\ord(f)=2$ and we have the following lemma.
   \begin{lemma}
     For $f\in\m$ the following are equivalent:
     \begin{enumerate}
     \item $\mu(f)=1$.
     \item $\tau(f)=1$.
     \item
       \begin{enumerate}
       \item If $\Char(\K)\not=2$, then $f\sim_r x_1^2+\ldots+x_n^2$.
       \item If $\Char(\K)=2$, then $n=2k$ is even and $f\sim_rx_1x_2+\ldots+x_{2k-1}x_{2x}$. 
       \end{enumerate}
     \item The same statement as in (c) for contact equivalence $\sim_c$.
     \end{enumerate}
   \end{lemma}
   \begin{proof}
     This follows from \cite[3.5, Prop.~3]{GK90}.
   \end{proof}

   The class
   of \emph{quasihomogeneous singularities}, i.\ e. of
   singularities with a quasihomogeneous polynomial representative under right (or contact)
   equivalence, is an important class of singularities in
   characteristic zero. 
   
   In positive characteristic we have to be more careful, since the
   Euler relation is not helpful when the characteristic divides the
   weighted degree. E.g.\ $f=x^p+y^{p-1}$ is quasihomogeneous of
   degree $p\cdot (p-1)$ with respect to $w=(p-1,p)$ with
   $\tau(f)=p\cdot (p-2)$ and $\mu(f)=\infty$ if $\Char(\K)=p$. However, when the
   characteristic does not divide the weighted degree some of the good
   properties still hold true.

   \begin{proposition}\label{prop:qh}
     Let $f\in\K[\ux]\setminus \K$ be \qh\  of type $(w;d)$ with $\gcd(w_1,\ldots,w_n)=1$.
     \begin{enumerate}
     \item If $f\in\m^3$ then
       \begin{displaymath}
         \mu(f)<\infty\;\;\;\Longleftrightarrow\;\;\;\tau(f)<\infty\text{ and }\Char(\K)\not|\;d.
       \end{displaymath}
       In this case obviously $\mu(f)=\tau(f)$.
     \item If $\Char(\K)\not|\;d$ and $g\in\Kx$, then
       \begin{displaymath}
         f\;\sim_r\; g\;\;\;\Longleftrightarrow\;\;\;f\;\sim_c\;g.
       \end{displaymath}
     \end{enumerate}
   \end{proposition}
   \begin{proof}
     \begin{enumerate}
     \item If the characteristic does not divide $d$ and
       $\tau(f)<\infty$ then we are done
       by the Euler formula. Conversely, if $\mu(f)<\infty$ then
       $\tau(f)<\infty$ and we have to show that  the
       characteristic does not divide $d$. Assume the contrary. The
       Euler formula then gives the  identity
       \begin{displaymath}
         w_1\cdot x_1\cdot f_{x_1}+\ldots+w_n\cdot x_n\cdot f_{x_n}=0.
       \end{displaymath}
       Since $\gcd(w_1,\ldots,w_n)=1$ we may assume that $w_n$
       is not divisible by the characteristic, and we thus deduce
       \begin{displaymath}
         x_n\cdot f_{x_n}=-\frac{w_1}{w_n}\cdot x_1\cdot
         f_{x_1}-\ldots-\frac{w_{n-1}}{w_n}\cdot x_{n-1}\cdot
         f_{x_{n-1}}.
       \end{displaymath}
       $f$ being in $\m^3$ the variable $x_n$ is not zero in $M_f=\Kx/\jj(f)$, so
       that $f_{x_n}$ is a zero divisor in the $\Kx/\langle
       f_{x_1},\ldots,f_{x_{n-1}}\rangle$. Thus
       $f_{x_1},\ldots,f_{x_n}$ is not a regular sequence in the
       Cohen-Macaulay ring $\Kx$, and therefore the $\K$-algebra $M_f$
       is not zero-dimensional\lang{ (see e.g.\ \cite[Corollary
       B.8.3]{GLS07})}, i.e.\ we get the contradiction $\mu(f)=\infty$.
     \item The proof works as in characteristic zero since $d$-th roots
       exist in $\Kx^*$ if $d$ is not divisible by $\Char(\K)$, see
       e.g.\ \cite[Lemma~2.13]{GLS07}). 
     \end{enumerate}
   \end{proof}

   Note that the condition $f\in\m^3$ cannot be avoided\kurz{.}\lang{, though as
     seen in the proof it can be weakened to e.g.\
     \begin{displaymath}
       \exists\;i=1,\ldots,n\;:\Char(\K)\not|\;w_i\mbox{ and }\;x_i\not\in\langle
       f_{x_1},\ldots,f_{x_{i-1}},f_{x_{i+1}},\ldots,f_{x_n}\rangle.
     \end{displaymath}
     To see that we cannot avoid it completely }
   \kurz{ To see this } consider
   $f=xy\in\K[[x,y]]$ with $\Char(\K)=2$. It is \qh\  of type
   $\big((1,1);2\big)$ and the Milnor number is one, yet the
   characteristic divides the weighted degree.

   \medskip

   When classifying singularities with respect to right or contact
   equivalence the first classes one comes across
   have quasihomogeneous representatives. This is maybe the most
   important reason why they deserve attention. The next more
   complicated class of singularities are those which have a
   quasihomogeneous principal part that somehow governs its discrete
   part of the classification. 

   For a power series $f=\sum_\alpha a_\alpha\ux^\alpha\in\Kx$ and a
   weight vector $w\in\Z_{>0}^n$ we denote by
   \begin{displaymath}
     \In_w(f)=\sum_{w\cdot\alpha\text{ minimal}} a_\alpha\ux^\alpha
   \end{displaymath}
   the \emph{initial form} or \emph{principal part} of $f$ with
   respect to $w$. We call the power series $f$
   \emph{right semi-quasihomogeneous} \rsqh\  respectively 
   \emph{contact semi-quasihomogeneous} \csqh\  with
   respect to $w$ if $\mu\big(\In_w(f)\big)<\infty$ respectively 
   $\tau\big(\In_w(f)\big)<\infty$. A right resp.~contact equivalence class of
   singularities is called \emph{semi-quasihomogeneous}  if it has
   a semi-quasihomogeneous representative. Note that in characteristic
   zero the notions \rsqh\  and \csqh\  coincide. 
   
   Moreover, in characteristic zero it is known that semi-quasihomogeneous
   singularities are always isolated and that their Milnor number
   coincides with the Milnor number of the principal part\lang{, i.e.\ if
   $\K=\C$ their
   topology is governed by the principal part}. In positive
   characteristic we get an analogous statement.

   \begin{proposition}\label{prop:sqh}
     Let $f\in\Kx$ and $w\in\Z_{>0}^n$.
     \begin{enumerate}
     \item If $\mu\big(\In_w(f)\big)<\infty$ and $d=\deg_w\big(\In_w(f)\big)$, then 
       \begin{displaymath}
         \mu(f)=\mu(\In_w(f))=\left(\frac{d}{w_1}-1\right)\cdot\ldots\cdot\left(\frac{d}{w_n}-1\right)<\infty.
       \end{displaymath}
     \item $\tau(f)\leq\tau\big(\In_w(f)\big)$. 
     \end{enumerate}
     In particular, \lang{semi-quasihomogeneous singularities are
       isolated.}\kurz{if $f$ is \rsqh\ (resp.\ \csqh) then $f$
       (resp.\ $R_f$) is an isolated singularity.}
   \end{proposition}
   \begin{proof}
     \begin{enumerate}
     \item Let $d$ be the degree of  $\In_w(f)$. Then
       \begin{displaymath}
         f':=\frac{f\big(t^{w_1}x_1,\ldots,t^{w_n}x_n\big)}{t^d}=
         \In_w(f)+t\cdot g'\in\K[[\ux,t]]
       \end{displaymath}
       for some power series $g'\in\K[[\ux,t]]$.
       We can use $f'$ to define the following local
       $\K$-algebra homomorphism
       \begin{displaymath}
         \K[[\uz,t]]\longrightarrow\K[[\ux,t]]:t\mapsto t,z_i\mapsto
         f'_{x_i}.
       \end{displaymath}
       This gives $\K[[\ux,t]]$ the structure of a
       $\K[[\uz,t]]$-algebra, and if we tensorise 
       $\K[[\ux,t]]$ with $\K=\K[[\uz,t]]/\langle\uz,t\rangle$ we get
       \begin{displaymath}
         \;\;\;\;\;\K[[\ux,t]]\otimes_{\K[[\ux,t]]}\K=
         \K[[\ux,t]]/\langle f'_{x_1},\ldots,f'_{x_n},t\rangle
         \cong
         \Kx/\jj(\In_w(f)),
       \end{displaymath}
       which by assumption is a finite dimensional $\K$-vector space
       of dimension $\mu(\In_w(f))$. Since
       $(f'_{x_1},\ldots,f'_{x_n},t)$ is a regular sequence, 
       $\K[[\ux,t]]$ is flat as a $\K[[\uz,t]]$-module \lang{(see
       e.g.\ \cite[Theorem~18.16]{Eis96})}
       and thus it is free of rank $\mu\big(\In_w(f)\big)$ by Nakayama's
       Lemma. Tensoring with $\K[[t]]=\K[[\uz,t]]/\langle \uz\rangle$
       we get that
       \begin{equation}\label{eq:sqh:1}
         \K[[\ux,t]]\otimes_{\K[[\uz,t]]}\K[[\uz,t]]/\langle\uz\rangle
         \cong \K[[\ux,t]]/\langle f'_{x_1},\ldots,f'_{x_n}\rangle
       \end{equation}
       is a free $\K[[t]]$-module of rank $\mu\big(\In_w(f)\big)$. 
       
       Passing to the field of fractions $\LL=\K((t))$ of $\K[[t]]$ we have the
       isomorphism of local $\LL$-algebras
       \bmath
         \varphi:\LL[[\ux]]\longrightarrow\LL[[\ux]]:x_i\mapsto t^{w_i}x_i.
       \emath
       Moreover
       \bmath
         f'=\frac{\varphi(f)}{t^d},
       \emath
       so that in $\LL[[\ux]]$ we have the equality of ideals
       \bmath
         \langle f'\rangle=\langle \varphi(f)\rangle
         \lang{\;\;\;}\text{ and }\lang{\;\;\;}
         \jj(f')=\jj\big(\varphi(f)\big)=\varphi\big(\jj(f)\big).
       \emath
       Extending scalars in \eqref{eq:sqh:1} to the field of fractions
       $\LL$ we get an isomorphism of $\LL$-vector spaces
       \begin{displaymath}
         \K[[\ux,t]]/\langle f'_{x_1},\ldots,f'_{x_n}\rangle\otimes_{\K[[t]]}\LL
         \cong \LL[[\ux]]/\jj(f')
         \cong \LL[[\ux]]/\jj(f).
       \end{displaymath}
       By freeness the left hand side is of dimension
       $\mu\big(\In_w(f)\big)$ while the right hand side has dimension
       $\mu(f)$. For the formula for $\mu(\In_w(f))$ see \cite[Prop.~3.8]{BGM10}.
     \item It suffices to consider the case
       $\tau\big(\In_w(f)\big)<\infty$, and the proof then is similar to (a), using the map
       \begin{displaymath}
         \K[[\uz,t]]\longrightarrow\K[[\ux,t]]:t\mapsto t,z_0\mapsto
         f', z_i\mapsto f'_{x_i}
       \end{displaymath}
       for $i=1,\ldots,n$, where $\uz=(z_0,\ldots,z_n)$. Since
       $(f',f'_{x_1},\ldots,f'_{x_n},t)$ is not a regular sequence,
       $\K[[\ux,t]]/\langle f',f'_{x_1},\ldots,f'_{x_n}\rangle$ is
       finitely generated but may have torsion as a
       $\K[[t]]$-module. Tensoring this module with $\K$ over
       $\K[[t]]$ gives a $\K$-vector space of dimension
       $\tau(\In_w(f))$. Tensoring with $\LL=\K((t))$ over $\K[[t]]$
       kills the torsion and gives an $\LL$-vector space of dimension
       $\tau(f)\leq \tau(\In_w(f))$.
     \end{enumerate}
   \end{proof}

   Note that the condition on the finiteness of $\mu\big(\In_w(f)\big)$ in
   Proposition~\ref{prop:sqh} (a) cannot be avoided, and $\tau(f)$
   will in general not coincide with $\tau\big(\In_w(f)\big)$.  
   Moreover, if $f\in\m^3$ we get from Proposition \ref{prop:qh} that
   $\mu(\In_w(f))<\infty$ is equivalent to $\tau(\In_w(f))<\infty$ and
   $\Char(\K)\nmid d=\deg_w(\In_w(f))$. 

   \lang{ 
     Consider $f=x^7+x^6y+y^4\in\K[[x,y]]$ with $\Char(\K)=7$. $f$ is
     \csqh\  with principal part $\In_w(f)=x^7+y^4$ which is \qh\  of type
     $\big((4,7);28\big)$ with
     \bmath
     \tau\big(\In_w(f)\big)=21>17=\tau(f).
     \emath
     Moreover,
     $\mu\big(\In_w(f)\big)=\infty$, but $\mu(f)=21$. Note, that here of course
     the characteristic of the base field divides the weighted degree
     of $\In_w(f)$.
   }

   \section{Piecewise filtrations and graded algebras}\label{sec:piecewisefiltration}

   Fixing a weight vector $w\in\Z_{>0}^n$ we get in a natural way
   a filtration on $\Kx$. If a singularity is semi-quasihomogeneous
   with respect to $w$ then this 
   filtration is perfectly suited to study the singularity and in
   general $w$ singles out a unique facet of the Newton diagram of
   the defining power series.
   However, in general we will have to consider more complicated
   filtrations since there is no single facet of the Newton
   diagram which captures enough information on the singularity. This was noted by
   Arnold and he introduced in \cite{Arn75}
   piecewise filtrations which are used to study non-degeneracy
   conditions by Kouchnirenko in \cite{Kou76}. 

   Given weight vectors $w_i\in\Q_{>0}^n$ with positive entries,
   $i=1,\ldots,k$, they define linear functions
   \begin{displaymath}
     \lambda_i:\R^n\longrightarrow\R:r\mapsto w_i\cdot r:=\sum_{j=1}^n
     w_{i,j}\cdot r_j,
   \end{displaymath}
   and their minimum defines a convex piecewise linear function
   \begin{displaymath}
     \lambda:\R^n\longrightarrow\R:r\mapsto\min\{\lambda_1(r),\ldots,\lambda_k(r)\}.
   \end{displaymath}
   We will always assume that the set of weights is
   \emph{irredundant}, i.e. that none of the $\lambda_i$ is superfluous
   in the definition of $\lambda$.
   The set
   \begin{displaymath}
     P_\lambda=\{r\in\R_{\geq 0}^n\;|\;\lambda(r)=1\}
   \end{displaymath}
   is a compact rational polytope of dimension $n-1$ in the positive orthant $\R_{\geq 0}^n$,
   and its facets are given by
   \begin{displaymath}
     \Delta_i=\{r\in P_\lambda\;|\;\lambda_i(r)=1\}.
   \end{displaymath}
   $P_\lambda$ has the property that each ray in $\R_{\geq 0}^n$ emanating
   from the origin meets $P_\lambda$ in precisely one point and that
   the region in $\R_{\geq 0}^n$ lying above $P_\lambda$ is convex. Following
   the convention of Wall (see \cite{Wal99}) we call such polytopes
   \emph{$C$-polytopes}. \label{page:cpolytope} Thus, irredundant sets of weight vectors
   define $C$-polytopes. 

   Conversely, given a $C$-polytope $P$ the suitably
   scaled inner normal vectors of its facets define an irredundant set
   of weight vectors such that $P=P_\lambda$ for the  corresponding piecewise
   linear function $\lambda$. We denote by $\lambda_P$ the piecewise linear function
   defined by $P$, and by $\lambda_\Delta$ the linear function
   corresponding to a facet  $\Delta$ of $P$.

   \smallskip

     To each power series $f=\sum_\alpha a_\alpha\ux^\alpha\in\Kx$ we
     can associate its \emph{Newton polyhedron}
     $\Gamma_+(f)$ as the convex hull of the set
     \bmath
       \bigcup_{\alpha\in\supp(f)}\big(\alpha+\R_{\geq 0}^n\big)
     \emath
     where $\supp(f)=\{\alpha\;|\;a_\alpha\not=0\}$ denotes the
     support of $f$. \lang{This is an unbounded polyhedron in $\R^n$.}
     Following Arnol'd we call
     the union $\Gamma(f)$ of its compact faces the \emph{Newton
       diagram} of $f$, some authors call it the \emph{Newton polytope}
     resp.~ the \emph{Newton polygon} if $n=2$. Note that the Newton diagram of a \qh\ or
     \sqh\ polynomial has exactly one facet, where a facet is a
     face of dimension $n-1$. For later use we denote by $\Gamma_-(f)$ the
     union of line segments joining points on $\Gamma(f)$ with the
     origin. (See Figure~\ref{fig:np} for an example.)      

     \begin{figure}[h]
       \centering
       \setlength{\unitlength}{0.35mm}
       \begin{tabular}{ccc}
         \begin{picture}(100,80)
           \put(10,10){\line(1,0){100}}
           \put(10,20){\line(1,0){100}}
           \put(10,30){\line(1,0){100}}
           \put(10,40){\line(1,0){100}}
           \put(10,50){\line(1,0){100}}
           \put(10,60){\line(1,0){100}}
           \put(10,70){\line(1,0){100}}         
           \put(10,10){\line(0,1){70}}
           \put(20,10){\line(0,1){70}}
           \put(30,10){\line(0,1){70}}
           \put(40,10){\line(0,1){70}}
           \put(50,10){\line(0,1){70}}
           \put(60,10){\line(0,1){70}}
           \put(70,10){\line(0,1){70}}
           \put(80,10){\line(0,1){70}}
           \put(90,10){\line(0,1){70}}
           \put(100,10){\line(0,1){70}}
           \thicklines\drawline[12](20,50)(40,30)
           \thicklines\drawline[12](40,30)(80,10)         
           \linethickness{0.2mm}\scriptsize
           \Thicklines
           \put(20,50){\line(0,1){30}}
           \put(80,10){\line(1,0){30}}
           \put(20,70){\line(1,1){10}}
           \put(20,60){\line(1,1){20}}
           \put(20,50){\line(1,1){30}}
           \put(25,45){\line(1,1){35}}
           \put(30,40){\line(1,1){40}}
           \put(35,35){\line(1,1){45}}
           \put(40,30){\line(1,1){50}}
           \put(46.5,26.5){\line(1,1){52.5}}
           \put(53.5,23.5){\line(1,1){55.5}}
           \put(60,20){\line(1,1){50}}
           \put(66.5,16.5){\line(1,1){43.5}}
           \put(72.5,12.5){\line(1,1){37}}
           \put(80,10){\line(1,1){30}}
           \put(90,10){\line(1,1){20}}
           \put(100,10){\line(1,1){10}}

           \put(20,50){\circle*{4}}
           \put(30,40){\circle*{4}}
           \put(40,30){\circle*{4}}
           \put(80,10){\circle*{4}}
           \put(50,30){\circle{4}}
           \put(35,-7){\mbox{${\Gamma}_{_+}(f)$}}
         \end{picture}
         &
         \begin{picture}(100,80)
           \put(10,10){\line(1,0){90}}
           \put(10,20){\line(1,0){90}}
           \put(10,30){\line(1,0){90}}
           \put(10,40){\line(1,0){90}}
           \put(10,50){\line(1,0){90}}
           \put(10,60){\line(1,0){90}}
           \put(10,70){\line(1,0){90}}
           \put(10,10){\line(0,1){70}}
           \put(20,10){\line(0,1){70}}
           \put(30,10){\line(0,1){70}}
           \put(40,10){\line(0,1){70}}
           \put(50,10){\line(0,1){70}}
           \put(60,10){\line(0,1){70}}
           \put(70,10){\line(0,1){70}}
           \put(80,10){\line(0,1){70}}
           \put(90,10){\line(0,1){70}}
           \thicklines\drawline[12](20,50)(40,30)
           \thicklines\drawline[12](40,30)(80,10)
           \put(20,50){\circle*{4}}
           \put(30,40){\circle*{4}}
           \put(40,30){\circle*{4}}
           \put(80,10){\circle*{4}}
           \put(25,-7){\mbox{ ${\Gamma(f)}$}}
         \end{picture}
         &
         \begin{picture}(100,80)
           \put(10,10){\line(1,0){90}}
           \put(10,20){\line(1,0){90}}
           \put(10,30){\line(1,0){90}}
           \put(10,40){\line(1,0){90}}
           \put(10,50){\line(1,0){90}}
           \put(10,60){\line(1,0){90}}
           \put(10,70){\line(1,0){90}}
           \put(10,10){\line(0,1){70}}
           \put(20,10){\line(0,1){70}}
           \put(30,10){\line(0,1){70}}
           \put(40,10){\line(0,1){70}}
           \put(50,10){\line(0,1){70}}
           \put(60,10){\line(0,1){70}}
           \put(70,10){\line(0,1){70}}
           \put(80,10){\line(0,1){70}}
           \put(90,10){\line(0,1){70}}
           \thicklines\drawline[12](20,50)(40,30)
           \thicklines\drawline[12](40,30)(80,10)
           \thicklines\drawline[12](10,10)(20,50)
           \linethickness{0.2mm}\scriptsize
           \Thicklines
           \put(10,10){\line(1,0){70}}
           \put(70,10){\line(1,1){3.3}}
           \put(60,10){\line(1,1){6.7}}
           \put(50,10){\line(1,1){10}}
           \put(40,10){\line(1,1){13.3}}
           \put(30,10){\line(1,1){16.7}}
           \put(20,10){\line(1,1){20}}
           \put(10,10){\line(1,1){25}}
           \put(13.1,23.1){\line(1,1){17}}
           \put(17,37){\line(1,1){8}}
           \put(20,50){\circle*{4}}
           \put(30,40){\circle*{4}}
           \put(40,30){\circle*{4}}
           \put(80,10){\circle*{4}}
           \put(35,-7)
           {\mbox{ ${\Gamma}_{_{-}}(f)$}}
         \end{picture}
       \end{tabular}
       
       \caption{The Newton diagram of $x\cdot(y^4+xy^3+x^2y^2-x^3y^2+x^6)$.}
       \label{fig:np}
     \end{figure}
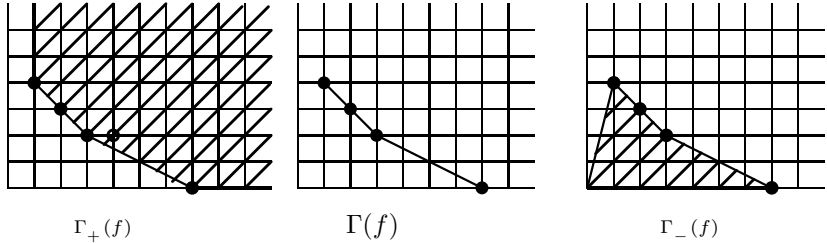

   Note that $\Gamma(f)$ is a $C$-polytope if and only if $f\in\Kx$ is
   a \emph{convenient} power series, i.e.\ if the 
   support of $f$ contains a point on each coordinate axis. $C$-Polytopes should
   thus be thought of as generalising the Newton diagram, and in our
   applications they will basically arise by extending Newton
   diagrams of non-convenient power series in a suitable way.

   For a $C$-polytope $P$ we denote by $N_P$ the lowest common multiple
   of the denominators of all entries in the weight vectors
   corresponding to $P$, so that $N_P\cdot \lambda_P$ takes non-negative
   integer values on $\N^n$. We then define a \emph{piecewise valuation} on $\Kx$ by
   \begin{displaymath}
     v_P(f):=\min\{N_P\cdot \lambda_P(\alpha)\;|\;a_\alpha\not=0\}\in\N
   \end{displaymath}
   for $0\not=f=\sum_{\alpha}a_\alpha \ux^\alpha\in\Kx$ and $v_P(0):=\infty$. $v_P$ satisfies
   \begin{displaymath}
     v_P(f\cdot g)\geq v_P(f)+v_P(g)
     \;\;\;\mbox{ and }\;\;\;
     v_P(f+g)\geq \min\{v_P(f),v_P(g)\}.
   \end{displaymath}
   Indeed we \lang{should like to point out that}\kurz{have}
   \begin{equation}
     \label{eq:vp=vd}
     v_P(f\cdot g)=v_P(f)+v_P(g)
     \;\;\;\Longleftrightarrow\;\;\;
     v_P(f)=v_\Delta(f)\;\;\mbox{ and }\;\;
     v_P(g)=v_\Delta(g)     
   \end{equation}
   for some facet $\Delta$ of $P$\tom{ (see e.g.\
     \cite[Lemma~2.1.22]{Bou09})}. 
   The sets
   \begin{displaymath}
     F_d:=F_{P,d}:=
     \{f\in\Kx\;|\;v_P(f)\geq d\}
   \end{displaymath}
   with $d\in\N$ are thus ideals in $\Kx$ and satisfy
   \bmath
     F_d\cdot F_e\subseteq F_{d+e},
   \emath
   i.e.\ they form a \emph{filtration} on $\Kx$. Note also that
   $F_0=\Kx$ and $F_1=\m$. Moreover, since all weight vectors
   corresponding to $P$ have only positive entries for each $d$ there
   is a positive integer $m$ such that 
   \begin{equation}
     \label{eq:filtration:1}
     \m^m\subseteq F_d,
   \end{equation}
   and also for any $k$ there is a $d$ such that each monomial of
   valuation degree $d$ has a degree bigger than or equal to $k$,
   i.e.\ such that
   \begin{equation}
     \label{eq:filtration:2}
     F_d\subseteq \m^k.
   \end{equation}
   Given any $C$-polytope $P$ and a power
   series $f\in\Kx$, we call the polynomial
   \begin{displaymath}
     \In_P(f)=\sum_{\stackrel{\lambda_P(\alpha) \text{ minimal}}{\alpha\in\supp(f)}}a_\alpha \ux^\alpha
   \end{displaymath}
   the \emph{initial form} or the \emph{principal part} of $f$ with
   respect to $P$. 
   $f$ is said to be \emph{piecewise homogeneous} \ph\  of
   degree $d\in\Q_{\geq 0}$ with respect to $P$ if $\lambda_P(\alpha)=d$ 
   for all $\alpha\in\supp(f)$. Note that then $f=\In_P(f)$ is a polynomial. 
   The power series $f$ is called 
   \emph{right semi-piecewise homogeneous} \rsph\  respectively
   \emph{contact semi-piecewise homogeneous} \csph\  with respect to $P$ if
   $\mu\big(\In_P(f)\big)<\infty$ respectively $\tau\big(\In_P(f)\big)<\infty$.

   Even though \ph, \rsph\  and \csph\  are straight forward generalisations of
   \qh, \rsqh\  and \csqh\  things get more complicated. One of the
   reasons is 
   that the product of two \ph\  polynomials need no longer be so,
   as  Example~\ref{ex:ph} shows.

   \begin{example}\label{ex:ph}
     Consider the weights $w_1=(1,2)$ and $w_2=(3,1)$ together with
     the polynomials $f=x^7+y^7$ and $g=x$. The
     corresponding $C$-polytope $P$ is the black polygon shown in
     Figure~\ref{fig:cpoly}. 
     \begin{figure}[h]
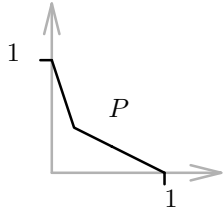

       \centering
       \begin{texdraw}
         \arrowheadtype t:V
         \drawdim cm \relunitscale 0.75 \linewd 0.04 
         \setgray 0.7
         \move (0 0) \avec (3 0)
         \move (0 0) \avec (0 3)
         \lpatt (2 0) \setgray 0
         \move (0 2) \lvec (0.4 0.8) \lvec (2 0)
         \setgray 0
         \move (-0.2 2) \lvec (0 2)
         \move (2 -0.2) \lvec (2 0)
         \htext (-0.8 2){$1$}
         \htext (2 -0.6){$1$}
         \htext (1 1){$P$}
       \end{texdraw}
       \caption{The $C$-polytope to $w_1=(1,2)$ and $w_2=(3,1)$.}
       \label{fig:cpoly}
     \end{figure}
     Both $f$ and $g$ are \ph\  with respect to $P$ of degree $7$
     respectively $1$. However,
     \begin{displaymath}
       \In_P(f\cdot g)=x^8
       \not=x^8+xy^7=f\cdot g\in F_{P,8},
     \end{displaymath}
     so that $f\cdot g$ is no longer piecewise homogeneous.     

     This example shows also that there cannot be any \emph{monomial
       ordering} $>$ (see \cite{GP08})
     which refines the piecewise degree with respect to $P$ if $P$ has
     more than one side. In fact, suppose there is, then either $x^7$ or $y^7$
     is the leading term of $f$. However, since $>$ refines the
     piecewise degree, $xf$ definitely will have
     leading term $x^8$ and $yf$ will have leading term $y^8$, in
     contradiction to the fact that the leading term must be compatible with the
     multiplication by monomials. This makes computations with
     piecewise filtrations difficult, in particular, we cannot use
     Gr\"obner basis methods directly.\hfill$\Box$
   \end{example}


   We also should like to point out, that a polynomial can be \ph\  with
   respect to many different $C$-polytopes. E.g.\ consider for
   $f=x^5+x^2y^2+y^5$ the two $C$-polytopes shown in Figure
   \ref{fig:C-polytopes}.
   \begin{figure}[h]
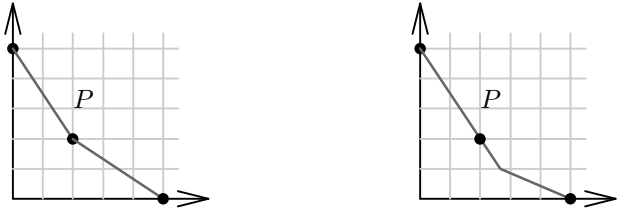

     \centering
     \begin{texdraw}
       \arrowheadtype t:V
       \drawdim cm \relunitscale 0.4 \linewd 0.06
       \setgray 0        
       \avec (6.5 0)
       \move (0 0) \avec (0 6.5)
       \setgray 0.8      
       \move (1 0) \lvec (1 5.5)
       \move (2 0) \lvec (2 5.5)
       \move (3 0) \lvec (3 5.5)
       \move (4 0) \lvec (4 5.5)
       \move (5 0) \lvec (5 5.5)
       \move (0 1) \lvec (5.5 1)
       \move (0 2) \lvec (5.5 2)
       \move (0 3) \lvec (5.5 3)
       \move (0 4) \lvec (5.5 4)
       \move (0 5) \lvec (5.5 5)
       \setgray 0.4
       \linewd 0.08
       \move (0 5) \fcir f:0 r:0.2
       \lvec (2 2)
       \move (2 2) \fcir f:0 r:0.2
       \lvec (5 0)
       \move (5 0) \fcir f:0 r:0.2
       \htext (2 3){$P$}
     \end{texdraw}
     \hspace{2.5cm}
     \begin{texdraw}
       \arrowheadtype t:V
       \drawdim cm \relunitscale 0.4 \linewd 0.06
       \setgray 0        
       \avec (6.5 0)
       \move (0 0) \avec (0 6.5)
       \setgray 0.8      
       \move (1 0) \lvec (1 5.5)
       \move (2 0) \lvec (2 5.5)
       \move (3 0) \lvec (3 5.5)
       \move (4 0) \lvec (4 5.5)
       \move (5 0) \lvec (5 5.5)
       \move (0 1) \lvec (5.5 1)
       \move (0 2) \lvec (5.5 2)
       \move (0 3) \lvec (5.5 3)
       \move (0 4) \lvec (5.5 4)
       \move (0 5) \lvec (5.5 5)
       \setgray 0.4
       \linewd 0.08
       \move (0 5) \fcir f:0 r:0.2
       \lvec (2.66 1)
       \lvec (5 0)
       \move (2 2) \fcir f:0 r:0.2
       \move (5 0) \fcir f:0 r:0.2
       \htext (2 3){$P$}
     \end{texdraw}
     \caption{Two $C$-polytopes w.r.t.\ which $x^5+x^2y^2+y^5$ is \ph.}
     \label{fig:C-polytopes}
   \end{figure}

   If $I\unideal\Kx$ is an ideal in $\Kx$ and $P$ is a $C$-polytope then
   the filtration induced by $P$ on $\Kx$ leads to the filtration
   \begin{displaymath}
     F_0+I/I\;\supseteq\;F_1+I/I\;\supseteq\;F_2+I/I\;\supseteq\;\ldots
   \end{displaymath}
   on $\Kx/I$, and induces thus the \emph{associated graded
     $\K$-algebra}
   \begin{displaymath}
     \gr_P\big(\Kx/I\big)
     =\bigoplus_{d\in\N} (F_d+I)/(F_{d+1}+I)\big)
     \cong\bigoplus_{d\in\N} F_d/\big((I\cap F_d)+F_{d+1}\big).
   \end{displaymath}
   The product of the classes of two monomials $\ux^\alpha$ and
   $\ux^\beta$ in $\gr_P(\Kx/I)$ satisfies
   \begin{equation}\label{eq:prodmon}
     \ux^\alpha\cdot\ux^\beta=\left\{
       \begin{array}[m]{ll}
         \ux^{\alpha+\beta}, & \mbox{ if
         }v_P(\ux^{\alpha+\beta})=v_P(\ux^\alpha)+v_P(\ux^\beta),\\
         0, & \mbox{ else.}
       \end{array}
     \right.
   \end{equation}

   We will show next that there are isomorphisms of vector spaces, 
   \begin{displaymath}
     M_f\cong \gr_P(M_f) 
     \;\;\;\mbox{ respectively }\;\;\;
     T_f\cong \gr_P(T_f),
   \end{displaymath}
   if the graded algebras are finite dimensional.
   Therfore, these graded algebras are natural means to study the
   singularity defined by $f$.  Arnol'd \cite{Arn75} has
   shown how to use a monomial basis of  $\gr_P(M_f)$ under
   suitable conditions on $f$ to compute a normal form for $f$. We
   will generalise this in Section~\ref{sec:normalforms}.

   \begin{proposition}\label{prop:grfinite}\label{prop:regularbasis}
     Let $I\unideal\Kx$ be an ideal and let $P$ be a $C$-polytope.
     \begin{enumerate}
     \item Then 
       \bmath
       \dim_\K(\gr_P(\Kx/I))=\dim_\K(\Kx/I).
       \emath
     \item 
       If $\dim_\K(\gr_P(\Kx/I))<\infty$,
       then any monomial basis of
       $\gr_P(\Kx/I)$ is a basis $\Kx/I$ as $\K$-vector space. 
     \end{enumerate}
   \end{proposition}
   \begin{proof}
     \begin{enumerate}
     \item The sequence of ideals
       \begin{displaymath}
         \Kx=F_0+I\supseteq F_1+I\supseteq \ldots\supseteq
         F_d+I\supseteq F_{d+1}+I\supseteq\ldots\supseteq I
       \end{displaymath}
       shows that $\dim_K(\Kx/I)<\infty$ if and only if there are only
       finitely many $d$ such that $0<\dim_K(F_d+I/F_{d+1}+I)<\infty$,
       and this is equivalent to $\dim_K\big(\gr_P(\Kx/I)\big)<\infty$. In
       this case the dimensions obviously coincide.
     \item 
       Let $B$ be any set of monomials whose residue classes in
       $\gr_P(\Kx/I)$ form a $\K$-vector space basis. 
       We have to show that the residue classes of the elements of $B$ in
       $\Kx/I$ generate $\Kx/I$ as a $\K$-vector space.
       \\
       Let $f\in\Kx$ be given and let $d=v_P(f)$ be its piecewise
       valuation. We then can write $f$ as
       \begin{displaymath}
         f=\sum_{\stackrel{\ux^\alpha\in B}{v_P(\ux^\alpha)=d}} c_\alpha
         \ux^\alpha + g_d+h_d       
       \end{displaymath}
       with $c_\alpha\in\K$, $g_d\in I\cap F_d$ and $h_d\in F_{d+1}$.
       \\
       We continue with $h_d$ in the same way, and thus for any
       $k\geq d$ there are $c_\alpha\in\K$, $g_k\in I\cap F_k$ and $h_k\in F_{k+1}$ such
       that
       \begin{displaymath}
         f=\sum_{\stackrel{\ux^\alpha\in B}{d\leq
             v_P(\ux^\alpha)\leq k}}c_\alpha \ux^\alpha 
         +(g_d+g_{d+1}+\ldots+g_k) +h_k,
       \end{displaymath}
       where $g=g_d+g_{d+1}+\ldots+g_k\in I$.
       Since $\dim_\K\big(\gr_P(\Kx/I)\big)<\infty$ there is a $d_0$ such that
       \bmath
       (I\cap F_d)+F_{d+1}=F_d\;\;\;\mbox{ for all }\;d\geq d_0,
       \emath
       i.e.\ $B\cap F_{d_0}=\emptyset$.
       For $k\geq d_0$ we thus have
       \begin{displaymath}
         f-\sum_{\stackrel{\ux^\alpha\in B}{d\leq
             v_P(\ux^\alpha)< d_0}}c_\alpha \ux^\alpha 
         =(g_d+g_{d+1}+\ldots+g_k) +h_k\in I+F_{k+1},
       \end{displaymath}
       where the left hand side does not depend on $k$.
       Using Krull's Intersection Theorem \tom{ (see
         e.g.\ \cite[Corollary~10.18]{AM69}) } this shows that
       \begin{displaymath}
         f-\sum_{\stackrel{\ux^\alpha\in B}{d\leq
             v_P(\ux^\alpha)< d_0}}c_\alpha \ux^\alpha \in
         \bigcap_{k\geq 0} (I+F_k)\stackrel{\eqref{eq:filtration:1},\eqref{eq:filtration:2}}{=}
         \bigcap_{k\geq 0} (I+\m^k)=I,
       \end{displaymath}
       and hence the claim.
     \end{enumerate}
   \end{proof}

   If we apply Proposition~\ref{prop:grfinite} to $M_f$ and $T_f$ we get the following
   corollary. 

   \begin{corollary}\label{cor:mutaugr}
     Let $f\in\Kx$ be a power series and $P$ a $C$-polytope.
     \begin{enumerate}
     \item $\mu(f)=\dim_\K(\gr_P(M_f))$.
     \item $\tau(f)=\dim_\K(\gr_P(T_f))$. 
     \end{enumerate}
   \end{corollary}

   \medskip

   \lang{
     $v_P$ induces a $\K$-linear decomposition of the polynomial ring
     $\K[\ux]=\bigoplus_{d\geq 0} \K[\ux]_d$ with
     \begin{displaymath}
       \K[\ux]_d=\langle \ux^\alpha\;|\;v_P(\ux^\alpha)=d\rangle_\K,
     \end{displaymath}
     and for any ideal $J\unideal\Kx$ we can consider
     \begin{displaymath}
       J_d=\K[\ux]_d\cap J.
     \end{displaymath}
     Note that in general $J\not=\bigoplus_{d\geq 0}J_d$ if the
     $C$-polytope $P$ has more than one facet, even if $J$ is generated
     as an ideal by piecewise homogeneous elements. See e.g.\ $J=\langle
     \In_P(h\cdot f)\;|\;h\in\K[[x,y]]\rangle$ with $P$ and $f$ as in
     Example~\ref{ex:ph} then it is easy to see that $x\cdot
     f=x^8+xy^7=f_8+f_{10}\in J$ but $f_{10}=xy^7\not\in J$. 
     
     \begin{proposition}\label{prop:grin}
       Let $I\unideal\Kx$ be an ideal in $\Kx$ and let $P$ be a
       $C$-polytope. Then there is a natural isomorphism of $\K$-vector spaces
       \begin{displaymath}
         \gr_P(\Kx/I)\cong\bigoplus_{d\geq 0}\K[\ux]_d/\In_P(I)_d,
       \end{displaymath}
       where $\In_P(I)=\langle \In_P(f)\;|\;f\in I\rangle$ is the
       initial ideal of $I$ with respect to $P$.
     \end{proposition}
     \begin{proof}
       Consider the $\K$-linear map
       \begin{displaymath}
         \varphi_d:\K[\ux]_d\longrightarrow F_d/\big((I\cap
         F_d)+F_{d+1}\big):f\mapsto \overline{f}
       \end{displaymath}
       sending a polynomial $f$ to its residue class. This map is
       obviously surjective, and we claim that
       $\ker(\varphi_d)=\In_P(I)_d$. If $f=\In_P(g)\in\In_P(I)_d$ with
       $g\in I$ then $f-g\in F_{d+1}$ and thus
       $\varphi_d(f)=\overline{g+(f-g)}=0$. If $f\in\ker(\varphi_d)$
       then $f=g+h$ with $g\in I\cap F_d$ and $h\in F_{d+1}$, so that
       $f=\In_P(g)\in\In_P(I)_d$. Thus $\varphi_d$ induces an
       isomorphism as desired.
     \end{proof}
     
     \begin{remark}\label{rem:inP}
       If the $C$-polytope $P$ in Proposition~\ref{prop:grin} has only
       \emph{one} facet, i.e.\ $P$ induces a weighted filtration, then we have
       a natural isomorphism 
       \begin{displaymath}
         \gr_P(\Kx/I)\cong\Kx/\In_P(I).
       \end{displaymath}
       The reason for this is that if $P$ has only one facet then
       \begin{displaymath}
         \In_P(I)=\bigoplus_{d\geq 0}\In_P(I)_d,
       \end{displaymath}
       since a weighted homogeneous polynomial lies in the weighted
       homogeneous ideal $\In_P(I)$ if and only if its weighted
       homogeneous summands belong to $\In_P(I)$.
       The isomorphism is
       thus induced by
       \begin{displaymath}
         \Kx\longrightarrow \gr_P(\Kx/I):f=\sum_d f_d\mapsto \sum_d \overline{f_d},
       \end{displaymath}
       where $f=\sum_d f_d$ is the decomposition of $f$ into its
       weighted homogeneous parts and $\overline{f_d}$ is the the
       residue class of $f_d$ in $F_d/\big((I\cap f_d)+F_{d+1}\big)$. 
       
       This fact can be used to compute a monomial basis for
       $\gr_P(\Kx/I)$. $P$ determines a weight vector $w$ and we can fix
       a local weighted degree ordering with respect to this weight
       vector $w$. If we then compute a standard basis of $I$ with
       respect to this ordering, the $w$-initial forms of the basis
       elements generate $\In_P(I)$. Moreover, we can compute the
       standard monomials of of $\Kx/\In_P(I)$ via the leading ideal and
       they are a monomial basis of both, $\Kx/\In_P(I)$ and of
       $\gr_P(\Kx/I)$.
     \end{remark}
   }

   For any $C$-polytope $P$ the piecewise valuation $v_P$ on $\Kx$ can
   easily be extended to the $\Kx$-module $\Der_{\K}(\Kx)$ of
   derivations on $\Kx$. For this we define
   \begin{displaymath}
     v_P(\xi)=\min\{\lambda_P(\alpha-e_i)\;|\;a_{i,\alpha}\not=0\}\kurz{,
       \text{ where }}
   \end{displaymath}
   \lang{where}
   \begin{displaymath}
     \xi=\sum_{i=1}^n\sum_{\alpha\in\N^n} a_{i,\alpha}\cdot
     \ux^{\alpha} \cdot \partial_{x_i}\not=0
   \end{displaymath}
   and    
   where $e_i$ is the $i$-th standard basis vector of $\Z^n$, i.e.\ we
   naturally extend
   \bmath
     v_P(\ux^\alpha \partial_{x_i})=\lambda_P(\alpha-e_i)
   \emath
   to all derivations where the derivation $\partial_{x_i}$ lowers the exponent of $x_i$
   in $\ux^{\alpha}$ by one. Note that $v_P(\partial_{x_i})$ is negative.

   Straight forward computations show that $v_P$ then satisfies 
   \lang{(see e.g.\ \cite[Lemma~2.2.3]{Bou09})}
   \begin{equation}
     \label{eq:derivation:1}
     v_P(\xi f)\geq v_P(\xi)+v_P(f)
   \end{equation}
   for any $0\not=f\in\Kx$ and any $0\not=\xi\in\Der_{\K}(\Kx)$.
   Moreover (see \lang{\cite[Lemma~2.2.5]{Bou09} or} \cite[Lemma~6.6]{Arn75}), 
   if $f\in\m^2$ and $g_1,\ldots,g_n\in\m$ with $v_P(g_i)>v_P(x_i)$ then
   $\varphi:\Kx\longrightarrow\Kx:x_i\mapsto x_i+g_i$ is an
   isomorphism and
   \begin{equation}
     \label{eq:derivation:2}
     \varphi(f)=f+\xi f+ h
   \end{equation}
   where
   \begin{displaymath}
     \xi=\sum_{i=1}^n g_i\partial_{x_i}
     \;\;\;
     \mbox{ and }
     \;\;\;
     v_P(h)>v_P(\xi)+v_P(f).
   \end{displaymath}

   The fact that we do not always have
   $v_P(\xi f)=v_P(\xi)+v_P(f)$ is somewhat annoying and forces us
   to adapt the filtrations induced by $v_P$ on
   the ideals
   \begin{displaymath}
     \jj(f)=\{ \xi f\;|\;\xi\in\Der_\K(\Kx)\}\kurz{ \text{
         respectively }}
   \end{displaymath}
   \lang{respectively}
   \begin{displaymath}
     \tj(f)=\{g\cdot f+\xi f\;|\;g\in\Kx,\xi\in\Der_\K(\Kx)\}.
   \end{displaymath}
   In the following definitions we will restrict our attention in $\jj(f)\cap F_d$ respectively
   $\tj(f)\cap F_d$ to those elements whose valuation is  \emph{expected} to
   be at least $d$, avoiding those who do so simply by bad luck.

   \kurz{\begin{definition}
     \begin{enumerate}
     \item} 
     For $d\geq 0$ we define the ideals
     \begin{displaymath}
       \jj_P^A(f)_d:=\{ \xi f\;|\; v_P(\xi)+v_P(f)\geq d\}\ideal \Kx
     \end{displaymath}
     respectively
     \begin{displaymath}
       \tj_P^{AC}(f)_d:=\{ g\cdot f+\xi f\;|\; 
       \min\{v_P(g)+v_P(f),v_P(\xi)+v_P(f)\}\geq d\}\ideal \Kx.
     \end{displaymath}
   \kurz{\item }
     Replacing $\jj(f)\cap F_d$ resp.\ $\tj(f)\cap F_d$ in the
     definition of $\gr_P(M_f)$ resp.\ $\gr_P(T_f)$ by $\jj_P^A(f)_d$
     resp.\ $\tj_P^{AC}(f)_d$ we define the graded $\K$-algebras
     \begin{displaymath}
       \gr_P^A(M_f):=\bigoplus_{d\geq 0} F_d\big/\big(\jj_P^A(f)_d+F_{d+1}\big)
     \end{displaymath}
     respectively
     \begin{displaymath}
       \gr_P^{AC}(T_f):=\bigoplus_{d\geq 0} F_d\big/\big(\tj_P^{AC}(f)_d+F_{d+1}\big).
     \end{displaymath}
   \kurz{\item
     A monomial basis of $\gr_P^A(M_f)$ respectively $\gr_P^{AC}(T_f)$ 
     is called a \emph{regular basis} for $M_f$ respectively $T_f$. 
     \end{enumerate}
   \end{definition}}

   \kurz{\begin{remark}
     \begin{enumerate}
     \item} These notions were introduced for $M_f$ and $\jj(f)$
       implicitly by Arnol'd \cite{Arn75} and explicitly by Wall
       \cite{Wal99}. We introduce in this paper the modification to $T_f$ and
       $\tj(f)$ in order to treat contact equivalence. The upper index A
       refers to Arnol'd's condition \Aa\ and AC to our condition \AC\
       (see Definition \ref{def:AAAC}).
     \kurz{\item}    
     We obviously have the inclusions
       \begin{displaymath}
         \tj(f)\cap F_d\supseteq \tj_P^{AC}(f)_d \supseteq  \jj_P^A(f)_d  \subseteq
         \jj(f)\cap F_d,
       \end{displaymath}
       and hence canonical surjections
       \begin{displaymath}
         \gr_P^A(M_f)\twoheadrightarrow\gr_P(M_f),\;\;\;\;
         \gr_P^A(T_f)\twoheadrightarrow\gr_P(T_f).
       \end{displaymath}
     \kurz{\end{enumerate}
   \end{remark}}

   Due to Proposition~\ref{prop:regularbasis} this yields together
   with Corollary \ref{cor:mutaugr}  the
   following result.

   \begin{corollary}\label{cor:regularbasis}
     Let $f\in\Kx$ be a power series
     and let $P$ be a $C$-polytope.
     \begin{enumerate}
     \item Any monomial basis $B$ of $\gr_P^A(M_f)$ generates 
       $\gr_P(M_f)$, and if
       $\mu(f)<\infty$ then $B$ also
       generates $M_f$. In particular,
       \begin{displaymath}
         \mu(f)=\dim_\K(\gr_P(M_f))\leq\dim_K(\gr_P^A(M_f)).
       \end{displaymath}
     \item Any monomial basis $B$ of
       $\gr_P^{AC}(T_f)$ generates $\gr_P(T_f)$, and if
       $\tau(f)<\infty$ then $B$ also
       generates $T_f$. In particular,
       \begin{displaymath}
         \tau(f)=\dim_\K(\gr_P(T_f))\leq\dim_K(\gr_P^A(T_f)).
       \end{displaymath}
     \end{enumerate}     
   \end{corollary}

   \lang{Following Arnol'd \cite{Arn75} and Wall \cite{Wal99}, who
   considered this notion for $M_f$,  we call a
   monomial basis of $\gr_P^A(M_f)$ respectively $\gr_P^{AC}(T_f)$ 
   a \emph{regular basis} for $M_f$ respectively $T_f$. }

   We should point out that the finiteness of $\mu(f)$ respectively of
   $\tau(f)$ does \emph{not} suffice in general to guarantee the finite
   dimensionality of $\gr_P^A(M_f)$ respectively of
   $\gr_P^{AC}(T_f)$. The reason for this is that elements of
   valuation $d$ in $\jj(f)$ respectively in $\tj(f)$ may not be
   contained in $\jj_P^A(f)_d$ respectively in $\tj_P^A(f)_d$, as in
   the following example. 

   \begin{example}[$T_{45}$-Singularity in characteristic $2$]\label{ex:t45}
     Let $\Char(\K)=2$ and let the $C$-polytope $P$ be defined by the
     weights $w_1=(4,6)$ and $w_2=(5,5)$. The polynomial
     $f=x^5+x^2y^2+y^4$ is \ph\  of degree $20$ with respect to $P$ with
     Tjurina number $\tau(f)=16$. For $n\geq 4$ we have
     \begin{displaymath}
       y^{4n}=g\cdot f+\xi f\in \tj(f),
     \end{displaymath}
     where $g=y^{4n-4}+x^2y^{4n-6}+x^4y^{4n-8}$ is \qh\  of degree
     $20n-20=v_P(y^4n)-v_P(f)$ and $\xi=(x\cdot g+x^2y^{4n-6})\cdot
     \partial_x$. Thus $g$ 
     guarantees that $y^{4n}$ is indeed in
     $\tj_P^{AC}(f)_{20n}$, however,
     \begin{displaymath}
       v_P(\xi)=v_P(x^2y^{4n-6}\partial_x)=20n-25\;<\;v_P(y^{4n})-v_P(f)
     \end{displaymath}
     has a valuation which is too small. Moreover, we cannot do any
     better, i.e.\
     \bmath
       y^{4n}\not\in\tj_P^{AC}(f)+F_{20n+1}
     \emath
     and thus
     \bmath
       \dim_\K\big(\gr_P^{AC}(T_f)\big)=\infty.
     \emath
     See also Example \ref{ex:t45-2}.
     \hfill$\Box$
   \end{example}

   The following lemma shows that $\gr_P^A(M_f)$ and $\gr_P^{AC}(T_f)$
   depend only on the initial part of $f$ w.r.t.\ the $C$-polytope $P$.

   \begin{lemma}\label{lem:grpp}
     If $P$ is a $C$-polytope and $f\in\Kx$  then
     \begin{displaymath}
       \gr_P^A(M_f)=\gr_P^A(M_{\In_P(f)})
       \;\;\text{ and }\;\;
       \gr_P^{AC}(T_f)=\gr_P^{AC}(T_{\In_P(f)}).
     \end{displaymath}
   \end{lemma}
   \begin{proof}
     For this we write $f=\In_P(f)+h$ for some $h\in\Kx$ with
     $v_P(h)\geq v_P(f)+1$, and we note that for any
     $\xi\in\Der_\K(\Kx)$
     \begin{equation}\label{eq:grpp:5}
       v_P(\xi h)\geq v_P(\xi)+v_P(h)\geq v_P(\xi)+v_P(f)+1.
     \end{equation}
     In order to show $\gr_P^A(M_f)=\gr_P^A(M_{\In_P(f)})$ we have
     to show
     \begin{equation}\label{eq:grpp:6}
       \jj^A_P(f)_d+F_{d+1}=\jj^A_P\big(\In_P(f)\big)_d+F_{d+1}
     \end{equation}
     for all $d\geq 0$.

     If $g\in\jj^A_P(f)_d$ then there is a derivation $\xi$
     such that $g=\xi f$ with
     \begin{displaymath}
       d\leq v_P(\xi)+v_P(f)=v_P(\xi)+v_P\big(\In_P(f)\big).
     \end{displaymath}
     In view of \eqref{eq:grpp:5} we thus have
     \begin{displaymath}
       g=\xi \In_P(f)+\xi h\in \jj_P^A\big(\In_P(f)\big)_d+F_{d+1},
     \end{displaymath}
     which shows that the left hand side in \eqref{eq:grpp:6} is
     contained in the right  hand side.

     On the other hand, if $g\in \jj_P^A\big(\In_P(f)\big)_d$ then
     there is a derivative $\xi$ such that $g=\xi \In_P(f)$ with
     \begin{displaymath}
       d\leq v_P(\xi)+v_P\big(\In_P(f)\big)=v_P(\xi)+v_P(f).
     \end{displaymath}
     Again, in view of \eqref{eq:grpp:5} we thus have
     \begin{displaymath}
       g=\xi f-\xi h\in \jj_P^A\big(\In_P(f)\big)_d+F_{d+1},
     \end{displaymath}
     which shows that the left hand side in \eqref{eq:grpp:6} is
     contained in the right  hand side.

     The proof for $\gr_P^{AC}(T_f)=\gr_P^{AC}(T_{\In_P(f)})$ works analogously.
   \end{proof}

   \begin{corollary}\label{cor:qhgr}
     Let $f\in\K[\ux]$ be \qh\  of type $(w;d)$ 
     and let $P$ be the $C$-polytope defined by the single weight
     vector $w$.
     \begin{enumerate}
     \item Then $\Gamma(f)\subseteq P$,
       $\gr_P^A(M_f)=\gr_P(M_f)$ and
       $\gr_P^{AC}(T_f)=\gr_P(T_f)$
     \item If moreover $\mu(f)<\infty$ respectively $\tau(f)<\infty$,
       then a set of monomials is a $\K$-vector space basis
       for $\gr_P^A(M_f)$ respectively for $\gr_P^{AC}(T_f)$ if and only
       if it is one for $M_f$ respectively for $T_f$.
     \end{enumerate}
   \end{corollary}
   \begin{proof}
     Since  $P$ has only one side it induces a grading
     on $\K[\ux]$ and a \lang{homogeneous} filtration on $\Kx$.
     \begin{enumerate}
     \item We note that the partial derivative $f_{x_i}$ is \qh\  of type
       $(w;d-w_i)$ if it does not vanish. Thus the ideals $\jj(f)$ and
       $\tj(f)$ are generated by weighted homogeneous elements. This
       implies that
       \bmath
         (\jj(f)\cap F_k)+F_{k+1}=\jj_P^A(f)_k+F_{k+1}
       \emath
       and
       \bmath
         (\tj(f)\cap F_k)+F_{k+1}=\tj_P^{AC}(f)_k+F_{k+1}
       \emath
       for all $k\geq 0$ as required. \lang{Let us elaborate this argument for
       the ideal $\jj(f)$. 

       Suppose that $h=\sum_{i=1}^n g_i\cdot f_{x_i}\in\jj(f)$ is
       given. We can decompose the $g_i$ into their quasihomogeneous
       parts
       \begin{displaymath}
         g_i=\sum_{j\geq 0}g_{i,j}
       \end{displaymath}
       with $g_{i,j}$ \qh\  of type $(w;j)$. Then $h$ decomposes into
       quasihomogeneous parts $h=\sum_{j\geq 0} h_j$ with
       \begin{displaymath}
         h_j=\sum_{f_{x_i}\not=0} g_{i,j-d+w_i}\cdot f_{x_i}.
       \end{displaymath}
       If we now suppose that $h\in F_k$ then we can replace $g_{i,j}$
       by zero for $j<k+d-w_i$, i.e.\ we may assume that
       \begin{displaymath}
         g_i=\sum_{j\geq k+d-w_i} g_{i,j}\in F_{k+d-w_i}.
       \end{displaymath}
       Setting
       \begin{displaymath}
         \xi=\sum_{f_{x_i}\not=0}g_i\cdot \partial_{x_i}
       \end{displaymath}
       we have $h=\xi f$ and necessarily
       $v_P(\xi)+v_P(f)=v_P(h)=k$. Thus $h\in\jj(f)_k$.}       
     \item  By Corollary \ref{cor:regularbasis} any monomial basis $B$
       of $\gr_P^A(M_f)$  is a generating
       set of $M_f$. However, by (a) and Corollary
       \ref{cor:mutaugr} these vector spaces have the same
       dimension. Hence, $B$ is a basis of $M_f$. 

       For the converse we note that $\jj(f)$ 
       is generated by weighted homogeneous polynomials. If $B$
       is a monomial basis $M_f$ and $\ux^\beta$ is
       any monomial, then there are $c_\alpha\in\K$ and a $g\in\jj(f)$  such that
       \begin{displaymath}
         \ux^\beta=\sum_{\ux^\alpha\in B}c_\alpha\cdot\ux^\alpha+g,
       \end{displaymath}
       and all $\ux^\alpha$ as well as $g$ are weighted homogeneous
       polynomials of the same weighted degree as $\ux^\beta$. 
       In particular $g\in \jj(f)\cap F_d$ with $d={v_P(\ux^\beta)}$, and thus $\ux^\beta$ is a 
       linear combination of the elements of $B$ in $\gr_P(M_f)$. This
       shows that $B$ generates $\gr_P(M_f)=\gr_P^A(M_f)$, and since
       $M_f$ and $\gr_P(M_f)$ have the same dimension by Corollary \ref{cor:mutaugr} $B$
       must be a basis of $\gr^A_P(M_f)$.

       The proof for $\gr_P^{AC}(T_f)$ and $T_f$ works in the same way.
     \end{enumerate}
   \end{proof}

   \section{Normal forms}\label{sec:normalforms}

   To obtain normal forms of  power series which are not
   right semi-quasi\-homo\-ge\-neous the only
   known method was introduced by Arnol'd in \cite{Arn75} over the
   complex numbers and slightly generalised by Wall in \cite{Wal99}.
   It requires the principal part $\In_P(f)$ of the power series (with
   respect to some $C$-polytope $P$) to be an \emph{isolated} singularity and
   its Milnor algebra to have a \emph{finite regular basis}. Arnol'd
   actually gives a more restrictive condition but his proof shows
   that this suffices as was pointed out by Wall. We generalise
   Arnol'd's condition both in the strict and in the weak form to the
   situation of contact equivalence and derive normal forms for right
   as well as for contact equivalence in arbitrary characteristic. 

   \begin{definition}\label{def:AAAC}
     Let $P$ be a $C$-polytope and let $f\in\Kx$ be a power series.
     \begin{enumerate}
     \item       
       We say that 
       $f$ satisfies \emph{condition} \Aa\  w.r.t.\ $P$, if for any
       $g\in\jj(f)$ there exists a derivation $\xi\in\Der_\K(\Kx)$ such
       that
       \begin{displaymath}
         v_P(g)=v_P(\xi)+v_P(f)<v_P(g-\xi f).
       \end{displaymath}
     \item
       $f$ satisfies
       \emph{condition} \AAa\ w.r.t. $P$, if
       $\dim_\K\big(\gr_P^A(M_f)\big)<\infty$.
     \item
       We  say
       $f$ satisfies \emph{condition} \AC\  w.r.t.\ $P$, if for all $h\in\tj(f)$ there exists a $g\in
       \Kx$ and a derivation $\xi\in\Der_\K(\Kx)$ such that
       \begin{displaymath}
         v_P(h)=\min\{v_P(g)+v_P(f),v_P(\xi)+v_P(f)\}\;<\;v_P(h-g\cdot
         f-\xi f).
       \end{displaymath}
     \item 
       Finally, we say $f$ satisfies
       \emph{condition} \AAC\  w.r.t.\ $P$, if 
       $\dim_\K\big(\gr_P^{AC}(T_f)\big)<\infty$.
     \end{enumerate}
   \end{definition}

   \begin{remark}
     \begin{enumerate}
     \item Condition \Aa\ is due to Arnol'd and is equivalent to
       \begin{displaymath}
         \big(\jj(f)\cap F_d\big)+F_{d+1}
         =
         \jj_P^A(f)_d+F_{d+1},
         \;\;\;\mbox{ for all }d\geq 0,
       \end{displaymath}
       i.e.\ if
       \begin{displaymath}
         \gr_P(M_f)=\gr_P^A(M_f).
       \end{displaymath}
       i.e.\ if
       \begin{displaymath}
         \mu(f)=\dim_\K\big(\gr_P^A(M_f)\big).
       \end{displaymath}
     \item Condition \AAa\ (\AAa\ stands for \emph{almost \Aa}), which
       is obviously weaker than \Aa, is due to Wall \cite{Wal99} and
       says that the Milnor algebra of $f$ has a regular basis.
     \item While condition \Aa\ is meant to deal with right
       equivalence, our condition \AC\ deals with contact
       equivalence. \AC\ is equivalent to
       \begin{displaymath}
         \big(\tj(f)\cap F_d\big)+F_{d+1}
         =
         \tj_P^{AC}(f)_d+F_{d+1}
         \;\;\;\mbox{ for all }d\geq 0,
       \end{displaymath}
       i.e.\ if
       \begin{displaymath}
         \gr_P(T_f)=\gr_P^{AC}(T_f).
       \end{displaymath}
       i.e.\ if
       \begin{displaymath}
         \tau(f)=\dim_\K\big(\gr_P^{AC}(T_f)\big).
       \end{displaymath}
     \item \AAC\ means \emph{almost \AC}\ and says that the Tjurina
       algebra of $f$ has a regular basis.
     \end{enumerate}
   \end{remark}

   \lang{The naming of the above conditions should explain the \emph{superscripts} in
     $\gr_P^A(M_f)$ respectively in $\gr_P^{AC}(T_f)$ and in the
     corresponding ideals.}
   
   \begin{remark}
     The above equivalence of the characterisations of condition \Aa\
     respectively \AC\ uses Corollary \ref{cor:regularbasis}.
     Corollary~\ref{cor:mutaugr} and~\ref{cor:regularbasis} together
     with Example \ref{ex:e33}  show that for isolated
     singularities the \emph{almost} conditions are indeed strictly
     weaker. Moreover,
     \begin{equation}\label{eq:AAA}
       \mu(f)<\infty \text{ and }f \text{ satisfies \Aa\ } \;\;\; \Longrightarrow\;\;\;
       f \text{ satisfies \AAa\ } \;\;\; \Longrightarrow \;\;\; \mu(f)<\infty
     \end{equation}
     and
     \begin{equation}\label{eq:ACAAC}
       \tau(f)<\infty\text{ and }f\text{ satisfies \AC\ } \; \Longrightarrow \;
       f\text{ satisfies \AAC\ } \; \Longrightarrow \; \tau(f)<\infty.
     \end{equation}
     
     We point out that by Lemma~\ref{lem:grpp} 
     \begin{equation}\label{eq:AAfinf}
       f \mbox{ satisfies \AAa\  resp.\ \AAC\ }
       \;\;\;\Longleftrightarrow\;\;\;
       \In_P(f) \mbox{ satisfies \AAa\  resp.\ \AAC\ },
     \end{equation}
     i.e.\ the conditions \AAa\  and \AAC\  only depend on the principal
     part of $f$.
   \end{remark}

   We now formulate our main result on normal forms without refering
   to the conditions \AAa\  respectively \AAC, and we will see later
   how they come in useful. The statement for right equivalence in
   this form without refering to condition \Aa\  or \AAa\  was first stated
   over the complex numbers in \cite[Theorem~2.1]{Wal99}, but was already
   proved by Arnol'd in \cite[Theorem~9.5]{Arn75}. We generalise the
   statement to contact equivalence and give a different proof which works for
   any algebraically closed field making an ``Ansatz'' with power series.
   Recall that $\ord(f)$ denotes the order of the power series $f$.

   \begin{theorem}[Normal forms with respect to right equivalence]\label{thm:nfre}
     Let $f\in\m$, $P$ be a $C$-polytope and
     $B=\{\ux^\alpha\;|\;\alpha\in\Lambda\}$ a regular basis for
     $M_{\In_P(f)}$. 
     
     If $\m^{k+2}\subseteq\m^2\cdot\jj(f)$ then
     \begin{equation}\label{eq:nfre:1}
       f\;\sim_r\;\In_P(f)+\sum_{\alpha\in\Lambda_f} c_\alpha \ux^\alpha.
     \end{equation}
     for suitable $c_\alpha\in\K$, where $\Lambda_f$ is the finite set
     \begin{displaymath}
       \Lambda_f=\big\{\alpha\in\Lambda\;\big|\;\deg(\ux^\alpha)\leq
       2k-\ord(f)+2,\; v_P(\ux^\alpha)\geq v_P\big(f-\In_P(f)\big)\big\}.
     \end{displaymath}
   \end{theorem}

   \begin{theorem}[Normal forms with respect to contact equivalence]\label{thm:nfce}
     Let $f\in\m$, $P$ be a $C$-polytope and
     $B=\{\ux^\alpha\;|\;\alpha\in\Lambda\}$ a regular basis for
     $T_{\In_P(f)}$. 
     
     If $\m^{k+2}\subseteq\m\cdot\langle f\rangle+\m^2\cdot\jj(f)$ then
     \begin{equation}\label{eq:nfce:1}
       f\;\sim_c\;\In_P(f)+\sum_{\alpha\in\Lambda_f} c_\alpha \ux^\alpha.
     \end{equation}
     for suitable $c_\alpha\in\K$, where $\Lambda_f$ is the finite set
     \begin{displaymath}
       \Lambda_f=\big\{\alpha\in\Lambda\;\big|\;\deg(\ux^\alpha)\leq
       2k-\ord(f)+2,\; v_P(\ux^\alpha)\geq v_P\big(f-\In_P(f)\big)\big\}.
     \end{displaymath}
   \end{theorem}

   We will only prove Theorem~\ref{thm:nfce} since the proof of
   Theorem~\ref{thm:nfre} works along the same lines.

   \begin{proof}[Proof of Theorem~\ref{thm:nfce}]
     In the proof we will write $f_P$ instead of $\In_P(f)$ to shorten
     the notation.
     The basic idea is to construct a finite sequence $(f_i)_{i=0}^m$
     with $f_0=f$ such that $f_i\sim_c f$ for all $i$ and that
     \begin{displaymath}
       f_m\;\equiv\;f_P+\sum_{\alpha\in\Lambda_f} c_\alpha \ux^\alpha
       \;\big(\mod \m^{2k-\ord(f)+3}\big).
     \end{displaymath}
     We try to do so by eliminating terms in $f$ (piecewise) degree by
     degree. If we succeed then by
     \cite[Theorem~2.1]{BGM10} we have 
     \begin{displaymath}
       f\;\sim_c\;f_m\;\sim_c\;f_P+\sum_{\alpha\in\Lambda_f} c_\alpha \ux^\alpha
     \end{displaymath}
     as desired since $2k-\ord(f)+2$ is a bound for the
     determinacy of $f$.

     We start our construction by denoting by $g=\In_P(f-f_P)$ the principal part
     of $g_1=f-f_P$ with respect to $P$ and setting $h=f-f_P-g$. If we set
     $d_0=v_P(f)=v_P(f_P)$ and $d_1=v_P(f-f_P)>d_0$ then $f_P\in F_{d_0}$
     is \ph\  of degree $d_0$, $g\in F_{d_1}$ is \ph\  of degree $d_1$ and
     $h\in F_{d_1+1}$. Moreover, since $B$ is a $\K$-vector space
     basis of $\gr_P^{AC}(T_{f_P})$ we have
     \begin{displaymath}
       g=\sum_{\stackrel{\alpha\in \Lambda}{\lambda_P(\alpha)=d_1}}
       c_\alpha \ux^\alpha + b_0\cdot f_P+\xi f_P+h'
     \end{displaymath}
     for suitable
     \bmath
       c_\alpha\in\K,\;
       b_0\in\Kx,\; 
       \xi=\sum_{i=1}^n b_i\cdot\partial_{x_i}\in\Der_\K(\Kx),\;
       \mbox{ and } h'\in\Kx
     \emath
     satisfying
     \begin{equation}
       \label{eq:nfce:2}
       d_1=\min\{v_P(b_0)+d_0,v_P(\xi)+d_0\}
       \;\;\;\mbox{ and }\;\;\;
       v_P(h')>d_1.
     \end{equation}
     Form \eqref{eq:nfce:2} we deduce that
     \bmath
       v_P(b_0)\geq d_1-d_0>0 
     \emath
     and thus
     \begin{equation}
       \label{eq:nfce:3}
       b_0\in\m,
     \end{equation}
     and also
     \begin{equation}\label{eq:nfce:4}
       v_P(b_i)-v_P(x_i)=v_P(b_i\partial_{x_i})\geq v_P(\xi)\geq d_1-d_0>0.
     \end{equation}
     From \eqref{eq:derivation:2} we know that then
     \begin{displaymath}
       \varphi:\Kx\longrightarrow\Kx:x_i\mapsto x_i-b_i
     \end{displaymath}
     is a local $\K$-algebra isomorphism of $\Kx$.
     
     Moreover, applying
     \eqref{eq:derivation:2} to $\varphi(f_P)$ and to $\varphi(g)$ we
     get
     \begin{align*}
       \varphi(f)=&\varphi(f_P)+\varphi(g)+\varphi(h)\\
       =&f_P-\xi f_P+h_1+g-\xi g+h_2+\varphi(h)\\
       =&(1+b_0)\cdot f_P+\sum_{\stackrel{\alpha\in \Lambda}{\lambda_P(\alpha)=d_1}}
       c_\alpha \ux^\alpha +(h'+h_1+h_2+\varphi(h)-\xi g)
     \end{align*}
     with $h',h_1,h_2,\varphi(h),\xi g\in F_{d_1+1}$ taking
     \eqref{eq:derivation:1} and \eqref{eq:nfce:4} into account. 
     
     Since by \eqref{eq:nfce:3} $b_0\in\m$ we can multiply by the
     inverse of the unit $1+b_0$ which is of the form $1+b$ with
     $v_P(b)\geq d_1-d_0>0$ so that we get
     \begin{displaymath}
       (1+b)\cdot \varphi(f)=f_P+\sum_{\stackrel{\alpha\in \Lambda}{\lambda_P(\alpha)=d_1}}
       c_\alpha \ux^\alpha+g_2
       \;\;\;\mbox{ with }\;\;\;v_P(g_2)>d_1.
     \end{displaymath}
     Setting $f_1=(1+b)\cdot \varphi(f)$ we have $f_1\sim_c f$,
     and we can go on inductively treating $g_2$ as we have treated
     $g_1=f-f_P$ before. That way we construct power series
     \begin{displaymath}
       f_m=f_P+\sum_{\stackrel{\alpha\in \Lambda}{d_1\leq
           \lambda_P(\alpha)\leq d_m}}
       c_\alpha \ux^\alpha +g_{m+1}
     \end{displaymath}
     with $g_{m+1}\in F_{d_m}$ and $d_1<d_2<\ldots<d_m$, $m\geq 0$. By
     \eqref{eq:filtration:2} we eventually have that
     \bmath
       F_{d_m}\subseteq \m^{2k-\ord(f)+3},
     \emath
     and we are done.
   \end{proof}

   Theorem~\ref{thm:nfre} has as easy corollary the result of Arnol'd
   which he proved over the complex numbers in
   \cite[Theorem~9.5]{Arn75}, even though he used condition \Aa\  and
   $\mu\big(\In_P(f)\big)<\infty$\lang{ (see also
   \cite[Theorem~2.1]{Wal99})}.

   \begin{corollary}[Normal forms for right equivalence]\label{cor:nfre}
     Let $P$ be a $C$-polytope and $f\in\m$ be a power series such
     that $\In_P(f)$ satisfies \AAa\  then $f$ is finitely right
     determined, \rsph\  and
     \begin{displaymath}
       f\;\sim_r\;\In_P(f)+\sum_{\stackrel{\ux^\alpha\in
           B}{v_P(\ux^\alpha)>d}} c_\alpha \ux^\alpha
     \end{displaymath}
     for suitable $c_\alpha\in \K$, where $B$ is a finite regular
     basis for $M_{\In_P(f)}$ and $d=v_P(\In_P(f))$. 
   \end{corollary}
   \begin{proof}
     If $\In_P(f)$ satisfies \AAa\  then $f$ does so as well by
     Lemma~\ref{lem:grpp}. By \eqref{eq:AAA} thus $\mu(f)$ is finite
     and 
     therefore $\jj(f)$ contains some power of the maximal ideal. Hence we
     are done by Theorem~\ref{thm:nfre} since a regular basis for
     $M_{\In_P(f)}$ is finite due condition \AAa. 
   \end{proof}

   Using Lemma~\ref{lem:grpp}, \eqref{eq:ACAAC} and
   Theorem~\ref{thm:nfce} we get the analogous statement for contact 
   equivalence. 

   \begin{corollary}[Normal forms for contact equivalence]\label{cor:nfce}
     Let $P$ be a $C$-polytope and $f\in\m$ be a power series such
     that $\In_P(f)$ satisfies \AAC\  then $f$ is finitely contact
     determined, \csph\  and
     \begin{displaymath}
       f\;\sim_c\;\In_P(f)+\sum_{\stackrel{\ux^\alpha\in
           B}{v_P(\ux^\alpha)>d}} c_\alpha \ux^\alpha
     \end{displaymath}
     for suitable $c_\alpha\in \K$, where $B$ is a finite regular
     basis for $T_{\In_P(f)}$ and $d=v_P(\In_P(f))$. 
   \end{corollary}

   The proof of Theorem~\ref{thm:nfre} and~\ref{thm:nfce} actually
   gives a more precise bound on the determinacy if \AAa\ 
   respectively \AAC\  is fulfilled.

   \begin{corollary}[Finite determinacy bound for right equivalence]\label{cor:fdre}
     Let $P$ be a $C$-polytope, $f\in\m$ be a power series such
     that $\In_P(f)$ satisfies \AAa\  and let $B$ be a regular basis for
     $M_{\In_P(f)}$. Then
     \begin{displaymath}
       d:=\max\big\{v_P\big(\In_P(f)\big),v_P\big(\ux^\alpha\big)\;\big|\; \ux^\alpha\in B\}
     \end{displaymath}
     is finite and $f\sim_r g$ for any $g\in\Kx$ with $v_P(f-g)>d$.
     \lang{

     }In particular, if $\m^{k+1}\subseteq F_{d+1}$, then $f$ is right $k$-determined.
   \end{corollary}

   \begin{corollary}[Finite determinacy bound for contact equivalence]\label{cor:fdce}
     Let $P$ be a $C$-polytope, $f\in\m$ be a power series such
     that $\In_P(f)$ satisfies \AAC\  and let $B$ be a regular basis for
     $T_{\In_P(f)}$. Then
     \begin{displaymath}
       d:=\max\big\{v_P\big(\In_P(f)\big),v_P\big(\ux^\alpha\big)\;\big|\; \ux^\alpha\in B\}
     \end{displaymath}
     is fintie and $f\sim_c g$ for any $g\in\Kx$ with $v_P(f-g)>d$.
     \lang{

     }In particular, if $\m^{k+1}\subseteq F_{d+1}$, then $f$ is contact $k$-determined.
   \end{corollary}
   We present the proof for these corollaries only in the case of contact equivalence.
   \begin{proof}
     By Corollary~\ref{cor:nfce} $f$ is
     finitely determined, and thus by \cite[Theorem~2.5]{BGM10} some
     power of $\m$ lies in $\langle 
     f\rangle +\m\cdot \jj(f)$, so that we are in the situation of
     Theorem~\ref{thm:nfce}. The proof of Theorem~\ref{thm:nfce} shows that
     \begin{displaymath}
       f\;\sim_c\;\In_P(f)+\sum_{\stackrel{\ux^\alpha\in B}{v_P(\ux^\alpha)>v_P(\In_P(f))}}c_\alpha \ux^\alpha + g_l
     \end{displaymath}
     for suitable $c_\alpha\in\K$ and with $g_l\in F_l$ for $l$
     arbitrarily large. Moreover, in the 
     process of constructing the transformations we see that terms of
     piecewise valuation larger than $d$ do not have any influence on the
     coefficients $c_\alpha$ of the above normal form. Thus any power
     series $g$ which coincides with $f$ up to valuation $d$ will give
     the same normal form and is thus contact equivalent to $f$. 
   \end{proof}

   The determinacy bounds from Corollaries \ref{cor:fdre} resp.\
   \ref{cor:fdce} for power series satisfying conditions \AAa\ resp.\
   \AC\ are in general much better than those for arbitrary isolated
   singularities given in \cite{BGM10} (see Example \ref{ex:q10}),
   hence the conditions \AAa\  and \AAC\  are
   desirable. In the following we give many examples of
   power series satisfying these conditions. We will first consider
   quasihomogeneous polynomials.

   \begin{proposition}\label{prop:qhgr}
     If $f\in\K[\ux]$ is \qh\  of type $(w;d)$ and $P$ is the
     $C$-polytope defined by the single weight vector $w$, then $f$
     satisfies the conditions \Aa\  and \AC\  with respect to $P$.
   \end{proposition}
   \begin{proof}
     This was proved in from Corollary~\ref{cor:qhgr} (a).
   \end{proof}

   From \eqref{eq:AAA} and \eqref{eq:ACAAC} together with
   \eqref{eq:AAfinf} it follows that any power
   series with an isolated quasihomogeneous principal part satisfies
   \AAa\  and \AAC.


   \begin{corollary}\label{cor:sqh-aa}\label{cor:qhAAAAC}
     If $f\in\Kx$ is \rsqh\  respectively \csqh\ w.r.t.\ $w$ then $f$ is
     \AAa\  respectively \AAC\ w.r.t.\ the $C$-polytope defined by $w$.
   \end{corollary}

   \begin{example}[$T_{45}$-Singularity in characteristic $2$]\label{ex:t45-2}
     The condition in Corollary~\ref{cor:qhAAAAC} that the $C$-polytope has only
     \emph{one} facet, i.e.\ that the 
     principal part is quasihomogeneous, is essential. Let
     $\Char(\K)=2$ and $f=x^5+x^2y^2+y^4+x^3y^2\in\K[[x,y]]$. Then $f$
     is \csph\  with respect to $P=\Gamma(f)$ with principal part
     $\In_P(f)=x^5+x^2y^2+y^4$ and 
     $\tau\big(\In_P(f)\big)=16$. However, $\tau(f)=\infty$, which is an
     alternative proof of the fact $f$ is not \AAC, as we have
     already seen in Example~\ref{ex:t45}.\hfill$\Box$
   \end{example}

   \begin{remark}\label{rem:snnd}
     In \cite{Wal99} Wall introduces the notion of inner Newton
     non-degeneracy which turns out to be a sufficient condition for
     \AAa\  and \AAC. Let us recall the definition here.
     A face $\Delta$ of $P$ is called an \emph{inner face} if it is not
     contained in any coordinate hyperplane.
     Each point $q\in\K^n$ determines a coordinate hyperspace
     $H_q=\bigcap_{q_i=0}\{x_i=0\}\subseteq\R^n$ in $\R^n$.
     We call $f$ \emph{inner non-degenerate \SND\ along
       $\Delta$} if for no common zero $q$ of $\jj(\In_\Delta(f))$ the
     polytope $\Delta$ contains a point on $H_q$,
     and we call $f$ \emph{inner Newton non-degenerate} \SNND\
     w.r.t.\ $P$ if $f$ is non-degenerate of type \SND\ along each
     \emph{inner} face of $P$. 

     Inner Newton non-degeneracy can be formulated differently so
     that the connection to $\gr_P^A(M_f)$ is more evident.
     Each face $\Delta$ of the Newton diagram of $f$ determines
     a finitely generated semigroup $C_\Delta$ in $\Z^n$ by considering those
     lattice points which lie in the cone over $\Delta$ with the
     origin as base. This semigroup then determines a finitely
     generated $\K$-algebra
     $\K[C_\Delta]=\K[\ux^\alpha|\alpha\in C_\Delta]$ and
     a $\K[C_\Delta]$-module
     \begin{displaymath}
       D_\Delta=\langle
       \ux^\alpha\cdot\partial_{x_i}\;|\;\ux^\alpha\cdot\partial_{x_i}\ux^\gamma\in K[C_\Delta]\;\;\forall\;\gamma\in C_\Delta\rangle_{\K[\C_\Delta]}
     \end{displaymath}
     generated by monomial derivations which leave $\K[C_\Delta]$
     invariant. Applying all elements in $D_\Delta$ to $\In_\Delta(f)$
     leads to an ideal $J_\Delta$ in $\K[C_\Delta]$, and
     Wall then shows that (see \cite[Prop.~2.2]{Wal99})
     \begin{displaymath}
       \dim_\K(\K[C_\Delta]/J_{\Delta})<\infty
       \;\;\;\Longleftrightarrow\;\;\;
       f\mbox{ is \SND\ along all \emph{inner} faces of }\Delta.
     \end{displaymath}
     The rings $\K[C_\Delta]/J_\Delta$ can be stacked neatly in an
     exact sequence of complexes whose homology Wall uses to show that
     (see \cite[Prop.~2.3]{Wal99})
     \begin{displaymath}
       f\mbox{ is \SNND }\;\;\;\Longrightarrow\;\;\;
       \dim_\K\big(\gr_P^A(M_f)\big)<\infty.
     \end{displaymath}
     Wall's arguments use only standard facts from toric geometry and
     homolgoical algebra and do not depend on the characteristic of
     the base field. This proves  Theorem \ref{thm:snnd-aa}, which shows
     that inner Newton non-degenerate singularities possess good
     normal forms w.r.t.\ right equivalence and also w.r.t.\ contact
     equivalence (see Corollaries \ref{cor:nfre} and \ref{cor:nfce}). 

     We refer to \cite{Wal99} and
     \cite[Sec.~3]{BGM10} for more information on inner Newton
     non-degeneracy. 
     \hfill$\Box$
   \end{remark}

   \begin{theorem}[Wall, \cite{Wal99}]\label{thm:snnd-aa}
     If $f\in\Kx$ is \SNND\ w.r.t.\ $P$, then $f$ is \AAa\  and \AAC\
     w.r.t.\ $P$, and hence
     $$\tau(f)\leq\mu(f)<\infty.$$
   \end{theorem}

   \section{Examples}\label{sec:examples}

   In this section we apply the results of the previous sections to
   the classification of singularities of low modality in positive
   characteristic. A full classification of hypersurface singularities
   of right modality at most $2$ and of contact modality at most $1$
   is still missing in positive characteristic, although a big part of
   this classification was achieved in \cite{GK90} and \cite{Bou02}. 

   \begin{example}[$Q_{10}$-Singularity in characteristic $2$]\label{ex:q10}
     Let $\Char(\K)=2$ and assume that $f\in\K[[x,y,z]]$ is \csqh\  with
     respect to the $C$-polytope $P$ containing $\Gamma(f)$ and with
     principal part $\In_P(f)=x^2z+y^3+z^4$. Using \textsc{Singular}
     (\cite{DGPS10})
     we see that
     \begin{displaymath}
       B=\{1,x,y,z,xy,xz,yz,z^2,xyz,xz^2,yz^2,z^3,xyz^2,xz^3,yz^3,xyz^3\}
     \end{displaymath}
     is a $\K$-vector space basis of $T_{\In_P(f)}$. By Proposition
     \ref{prop:qhgr} 
     we see that $B$ is indeed a
     regular basis for $T_f$ and that $f$ is \AC\  with
     \begin{displaymath}
       \dim_\K\big(\gr^{AC}_P(T_f)\big)=\tau\big(\In_P(f)\big)=16.
     \end{displaymath}
     Corollary \ref{cor:nfce} then shows that
     \begin{displaymath}
       f\;\sim_c\; x^2z+y^3+z^4+c_1\cdot xyz^2 + c_2\cdot xz^3 +
       c_3\cdot yz^3+ c_4\cdot xyz^3
     \end{displaymath}
     for some $c_1,\ldots,c_4\in\K$. Moreover, using the weight vector
     $w=(9,8,6)$ to determine the filtration induced by $P$ then
     \begin{displaymath}
       \max\{v_P(f),v_P(b)\;|\;b\in B\}=35
     \end{displaymath}
     and an easy computation shows that $\m^6\in F_{36}$. Thus $f$ is
     contact $5$-determined, and this bound of determinacy is better than the
     one obtained from \cite[Theorem~2.1]{BGM10}, which
     would be $11$.
     \begin{figure}[h]
       \centering
       \setlength{\unitlength}{0.4mm}
       \begin{picture}(50,90)(-30,-30)
         \linethickness{0.1mm}\scriptsize
         \thicklines\dashline[+30]{3}(0,0)(38,0)
         \thicklines\dashline[+30]{3}(0,0)(0,40)
         \thicklines\dashline[+30]{3}(0,0)(-12,-12)
         \thicklines\drawline[12](38,0)(50,0)
         \thicklines\drawline[12](0,40)(0,60)
         \thicklines\drawline[12](-12,-12)(-30,-30)
         \thicklines\drawline[12](38,0)(0,40)
         \thicklines\drawline[12](-12,-12)(38,0)
         \thicklines\drawline[12](-12,-12)(0,40)
         \put(0,40){\circle*{4}}
         \put(-5,18){\circle*{4}}
         \put(37,0){\circle*{4}}
         \put(-34,-30){$x$}
         \put(50,-4){$y$}
         \put(-4,60){$z$}
       \end{picture}       
       \caption{The Newton diagram of $xz^2+y^3+z^4$.}
       \label{fig:npspace}
     \end{figure}
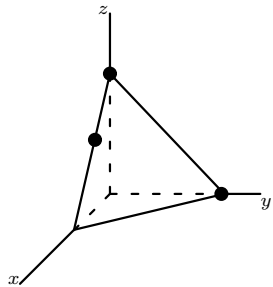
     \hfill$\Box$
   \end{example}

   When checking if certain monomials are zero in $\gr_P^A(M_f)$
   respectively $\gr_P^{AC}(T_f)$ the following lemma is very
   helpful. 

   \begin{lemma}\label{lem:cone}
     Let $P$ be a $C$-polytope, $\Delta$ a facet of $P$ and
     $f\in\Kx$. Moreover, denote by $C_\Delta$ the cone over
     $\Delta$, and assume that $\alpha,\beta\in
     C_\Delta\cap\Z^n$. 
     If $\ux^\alpha$ is zero in $\gr_P^A(M_f)$ respectively in
     $\gr_P^{AC}(T_f)$ then $\ux^{\alpha+\beta}$ is so.
   \end{lemma}
   \begin{proof}
     Since $\alpha$ and $\beta$ belong to the same cone $C_\Delta$
     Equation \eqref{eq:vp=vd} shows that
     \begin{displaymath}
       v_P\big(\ux^{\alpha+\beta}\big)=v_P\big(\ux^\alpha\big)+v_P\big(\ux^\beta\big).
     \end{displaymath}
     Thus in the graded algebra $\gr_P^A(M_f)$ respectively
     $\gr_P^{AC}(T_f)$ the class of $\ux^{\alpha+\beta}$ is the
     product of the classes of $\ux^{\alpha}$ and $\ux^{\beta}$ (see \eqref{eq:prodmon}). Since
     the former is zero by assumption so is the product. \tom{(See also\
     \cite[Lemma~3.2.3]{Bou09}.) }
   \end{proof}

   \begin{remark}
     With the notation and assumptions of Lemma~\ref{lem:cone} it
     follows that all monomials corresponding to the cone
     $\alpha+C_{\Delta}$ vanish in $\gr_P^A(M_f)$ respectively in
     $\gr_P^{AC}(T_f)$.
     \begin{figure}[h]
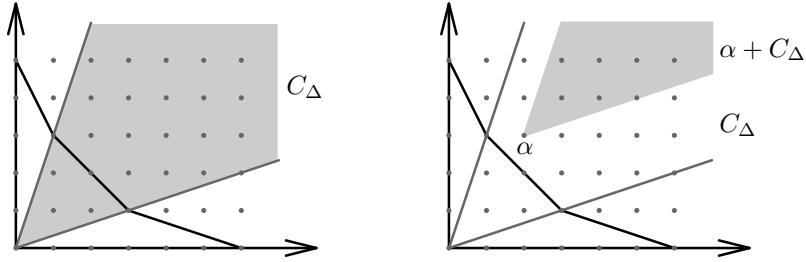

       \centering
       \begin{texdraw}
         \arrowheadtype t:V
         \drawdim cm \relunitscale 0.5 \linewd 0.06 
         \setgray 1         
         \move (0 0) \lvec (2 6) \lvec (7 6) \lvec (7 2.33) 
         \lvec (0 0) \lfill f:0.8
         \setgray 0
         \move (0 0) \avec (8 0)
         \move (0 0) \avec (0 6.5)
         \move (0 5) \lvec (1 3) \lvec (3 1) \lvec (6 0)
         \setgray 0.4
         \move (0 0) \lvec (2 6)
         \move (0 0) \lvec (7 2.33)
         \setgray 0
         \htext (7.2 4){$C_\Delta$}         
         \move (0 0) \fcir f:0.4 r:0.06 
         \move (1 0) \fcir f:0.4 r:0.06 
         \move (2 0) \fcir f:0.4 r:0.06 
         \move (3 0) \fcir f:0.4 r:0.06 
         \move (4 0) \fcir f:0.4 r:0.06 
         \move (5 0) \fcir f:0.4 r:0.06 
         \move (6 0) \fcir f:0.4 r:0.06 
         \move (0 1) \fcir f:0.4 r:0.06 
         \move (1 1) \fcir f:0.4 r:0.06 
         \move (2 1) \fcir f:0.4 r:0.06 
         \move (3 1) \fcir f:0.4 r:0.06 
         \move (4 1) \fcir f:0.4 r:0.06 
         \move (5 1) \fcir f:0.4 r:0.06 
         \move (6 1) \fcir f:0.4 r:0.06 
         \move (0 2) \fcir f:0.4 r:0.06 
         \move (1 2) \fcir f:0.4 r:0.06 
         \move (2 2) \fcir f:0.4 r:0.06 
         \move (3 2) \fcir f:0.4 r:0.06 
         \move (4 2) \fcir f:0.4 r:0.06 
         \move (5 2) \fcir f:0.4 r:0.06 
         \move (6 2) \fcir f:0.4 r:0.06 
         \move (0 3) \fcir f:0.4 r:0.06 
         \move (1 3) \fcir f:0.4 r:0.06 
         \move (2 3) \fcir f:0.4 r:0.06 
         \move (3 3) \fcir f:0.4 r:0.06 
         \move (4 3) \fcir f:0.4 r:0.06 
         \move (5 3) \fcir f:0.4 r:0.06 
         \move (6 3) \fcir f:0.4 r:0.06 
         \move (0 4) \fcir f:0.4 r:0.06 
         \move (1 4) \fcir f:0.4 r:0.06 
         \move (2 4) \fcir f:0.4 r:0.06 
         \move (3 4) \fcir f:0.4 r:0.06 
         \move (4 4) \fcir f:0.4 r:0.06 
         \move (5 4) \fcir f:0.4 r:0.06 
         \move (6 4) \fcir f:0.4 r:0.06 
         \move (0 5) \fcir f:0.4 r:0.06 
         \move (1 5) \fcir f:0.4 r:0.06 
         \move (2 5) \fcir f:0.4 r:0.06 
         \move (3 5) \fcir f:0.4 r:0.06 
         \move (4 5) \fcir f:0.4 r:0.06 
         \move (5 5) \fcir f:0.4 r:0.06 
         \move (6 5) \fcir f:0.4 r:0.06          
       \end{texdraw}
       \hspace{1.5cm}
       \begin{texdraw}
         \arrowheadtype t:V
         \drawdim cm \relunitscale 0.5 \linewd 0.06 
         \setgray 0.8         
         \move (2 3) \lvec (3 6) \lvec (7 6) \lvec (7 4.66) 
         \lvec (2 3) \lfill f:0.8
         \setgray 0
         \move (0 0) \avec (8 0)
         \move (0 0) \avec (0 6.5)
         \move (0 5) \lvec (1 3) \lvec (3 1) \lvec (6 0)
         \setgray 0.4
         \move (0 0) \lvec (2 6)
         \move (0 0) \lvec (7 2.33)
         \setgray 0
         \htext (1.8 2.5){$\alpha$}
         \htext (7.2 5){$\alpha+C_\Delta$}         
         \htext (7.2 3){$C_\Delta$}         
         \move (0 0) \fcir f:0.4 r:0.06 
         \move (1 0) \fcir f:0.4 r:0.06 
         \move (2 0) \fcir f:0.4 r:0.06 
         \move (3 0) \fcir f:0.4 r:0.06 
         \move (4 0) \fcir f:0.4 r:0.06 
         \move (5 0) \fcir f:0.4 r:0.06 
         \move (6 0) \fcir f:0.4 r:0.06 
         \move (0 1) \fcir f:0.4 r:0.06 
         \move (1 1) \fcir f:0.4 r:0.06 
         \move (2 1) \fcir f:0.4 r:0.06 
         \move (3 1) \fcir f:0.4 r:0.06 
         \move (4 1) \fcir f:0.4 r:0.06 
         \move (5 1) \fcir f:0.4 r:0.06 
         \move (6 1) \fcir f:0.4 r:0.06 
         \move (0 2) \fcir f:0.4 r:0.06 
         \move (1 2) \fcir f:0.4 r:0.06 
         \move (2 2) \fcir f:0.4 r:0.06 
         \move (3 2) \fcir f:0.4 r:0.06 
         \move (4 2) \fcir f:0.4 r:0.06 
         \move (5 2) \fcir f:0.4 r:0.06 
         \move (6 2) \fcir f:0.4 r:0.06 
         \move (0 3) \fcir f:0.4 r:0.06 
         \move (1 3) \fcir f:0.4 r:0.06 
         \move (2 3) \fcir f:0.4 r:0.06 
         \move (3 3) \fcir f:0.4 r:0.06 
         \move (4 3) \fcir f:0.4 r:0.06 
         \move (5 3) \fcir f:0.4 r:0.06 
         \move (6 3) \fcir f:0.4 r:0.06 
         \move (0 4) \fcir f:0.4 r:0.06 
         \move (1 4) \fcir f:0.4 r:0.06 
         \move (2 4) \fcir f:0.4 r:0.06 
         \move (3 4) \fcir f:0.4 r:0.06 
         \move (4 4) \fcir f:0.4 r:0.06 
         \move (5 4) \fcir f:0.4 r:0.06 
         \move (6 4) \fcir f:0.4 r:0.06 
         \move (0 5) \fcir f:0.4 r:0.06 
         \move (1 5) \fcir f:0.4 r:0.06 
         \move (2 5) \fcir f:0.4 r:0.06 
         \move (3 5) \fcir f:0.4 r:0.06 
         \move (4 5) \fcir f:0.4 r:0.06 
         \move (5 5) \fcir f:0.4 r:0.06 
         \move (6 5) \fcir f:0.4 r:0.06          
       \end{texdraw}
       
       \caption{The lattice points in $\alpha+C_\Delta$ correspond to
         monomials $x^\gamma$ vanishing in $\gr_P^A(M_f)$ respectively
         in $\gr_P^{AC}(T_f)$.}
       \label{fig:cone} 
     \end{figure}
   \end{remark}

   \begin{example}[$T_{pq}$-Singularities]\label{ex:tpq}
     Arnol'd considered in \cite[Example~9.6]{Arn75} the power series
     $f=x^p+\lambda x^2y^2+y^q\in\C[[\ux]]$ with $\lambda\not=0$ and
     $\frac{1}{p}+\frac{1}{q}<\frac{1}{2}$ or equivalently
     \begin{displaymath}
       pq-2p-2q>0.
     \end{displaymath}
     $f$ is \ph\  with respect to its Newton diagram $P=\Gamma(f)$
     depicted in Figure~\ref{fig:nparnold}. If we scale the corresponding weight
     vectors to length $2pq$ instead of one they are
     \begin{displaymath}
       w_1=(2q,pq-2q)\;\;\;\mbox{ and }\;\;\;w_2=(pq-2p,2p),
     \end{displaymath}
     and the piecewise degree of $f$ is then $\deg_P(f)=2pq$.
     \begin{figure}[h]
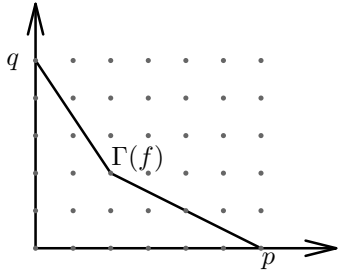

       \centering
       \begin{texdraw}
         \arrowheadtype t:V
         \drawdim cm \relunitscale 0.5 \linewd 0.06 
         \move (0 0) \avec (8 0)
         \move (0 0) \avec (0 6.5)
         \move (0 5) \lvec (2 2) \lvec (6 0)
         \setgray 0
         \htext (-0.8 4.7){$q$}
         \htext (6 -0.6){$p$}
         \htext (2 2){$\Gamma(f)$}         
         \move (0 0) \fcir f:0.4 r:0.06 
         \move (1 0) \fcir f:0.4 r:0.06 
         \move (2 0) \fcir f:0.4 r:0.06 
         \move (3 0) \fcir f:0.4 r:0.06 
         \move (4 0) \fcir f:0.4 r:0.06 
         \move (5 0) \fcir f:0.4 r:0.06 
         \move (6 0) \fcir f:0.4 r:0.06 
         \move (0 1) \fcir f:0.4 r:0.06 
         \move (1 1) \fcir f:0.4 r:0.06 
         \move (2 1) \fcir f:0.4 r:0.06 
         \move (3 1) \fcir f:0.4 r:0.06 
         \move (4 1) \fcir f:0.4 r:0.06 
         \move (5 1) \fcir f:0.4 r:0.06 
         \move (6 1) \fcir f:0.4 r:0.06 
         \move (0 2) \fcir f:0.4 r:0.06 
         \move (1 2) \fcir f:0.4 r:0.06 
         \move (2 2) \fcir f:0.4 r:0.06 
         \move (3 2) \fcir f:0.4 r:0.06 
         \move (4 2) \fcir f:0.4 r:0.06 
         \move (5 2) \fcir f:0.4 r:0.06 
         \move (6 2) \fcir f:0.4 r:0.06 
         \move (0 3) \fcir f:0.4 r:0.06 
         \move (1 3) \fcir f:0.4 r:0.06 
         \move (2 3) \fcir f:0.4 r:0.06 
         \move (3 3) \fcir f:0.4 r:0.06 
         \move (4 3) \fcir f:0.4 r:0.06 
         \move (5 3) \fcir f:0.4 r:0.06 
         \move (6 3) \fcir f:0.4 r:0.06 
         \move (0 4) \fcir f:0.4 r:0.06 
         \move (1 4) \fcir f:0.4 r:0.06 
         \move (2 4) \fcir f:0.4 r:0.06 
         \move (3 4) \fcir f:0.4 r:0.06 
         \move (4 4) \fcir f:0.4 r:0.06 
         \move (5 4) \fcir f:0.4 r:0.06 
         \move (6 4) \fcir f:0.4 r:0.06 
         \move (0 5) \fcir f:0.4 r:0.06 
         \move (1 5) \fcir f:0.4 r:0.06 
         \move (2 5) \fcir f:0.4 r:0.06 
         \move (3 5) \fcir f:0.4 r:0.06 
         \move (4 5) \fcir f:0.4 r:0.06 
         \move (5 5) \fcir f:0.4 r:0.06 
         \move (6 5) \fcir f:0.4 r:0.06 
       \end{texdraw}
       \caption{The Newton diagram of $x^p+\lambda x^2y^2+y^q$.}
       \label{fig:nparnold}
     \end{figure}
     Arnol'd describes in his paper a geometric procedure to compute a regular
     basis for $M_f$ if the partial derivatives have only two terms as
     in the example, and he deduces that $f$ satisfies condition \Aa\  and
     \AAa\  and  that
     \begin{displaymath}
       B=\{1,x,\ldots,x^p,y,y^2,\ldots,y^{q-1},xy\}
     \end{displaymath}
     is a regular basis for $M_f$. In particular,
     \begin{displaymath}
       \mu(f)=\dim_\K\big(\gr_P^A(M_f)\big)=p+q+1. 
     \end{displaymath}
     Since no monomial in $B$ lies above $\Gamma(f)$ it follows as seen
     in Corollary~\ref{cor:nfre} and~\ref{cor:fdre} that any power series
     whose principal part with respect to the above $P$ coincides with
     $T_{pq}$ actually is right equivalent to $T_{pq}$. 
     
     Arnol'd's arguments actually work for any field $\K$ where the
     \emph{characteristic} is neither two, nor divides $p$ or $q$. If the
     characteristic divides $p$ or $q$ then $\mu(f)=\infty$, and in the
     characteristic two case the Jacobian ideal is generated by the
     monomials $x^{p-1}$ and $y^{q-1}$, so that $\mu(f)=pq$.

     We now want to investigate $f=T_{pq}$ with respect to 
     contact equivalence and the condition \AC, and we first want to show
     that
     \begin{displaymath}
       \Char(\K)\not=2\;\;\;\Longrightarrow\;\;\;f \mbox{
         satisfies \AC\  and hence \AAC\ }.
     \end{displaymath}

     \emph{Assume first that in addition $\Char(\K)$ does neither divide $p$
     nor $q$ nor $pq-2\cdot (p+q)$.} Then $\mu(f)<\infty$ and thus also
     $\tau(f)<\infty$. Moreover, by Corollary~\ref{cor:regularbasis}
     the above $B$ generates $\gr_P^{AC}(T_f)$ and
     \begin{displaymath}
       T_f=\K[[x,y]]/\langle f,f_x,f_y\rangle.
     \end{displaymath}
     It is clear that other than $x^p$ the monomials in $B$ will stay
     linearly independent modulo $\tj(f)$, and all monomials
     $x^iy^j$ which are in $\jj(f)_d+F_{d+1}$ with
     $d=v_P(x^iy^j)$ are also in $\tj(f)_d+F_{d+1}$ (since they are a
     regular basis for $M_f$, see also \cite[Proposition~3.2.14]{Bou09}). To see that
     $f$ satisfies \AC\  it thus suffices to show that there are
     $a,b,c\in\K$ such that 
     \begin{displaymath}
       x^p=a\cdot f+b\cdot x\cdot f_x+c\cdot y\cdot f_y
     \end{displaymath}
     since then
     \begin{displaymath}
       x^p\in\tj(f)_{2pq}\subseteq\tj(f).
     \end{displaymath}
     Considering the coefficients for $x^p$, $x^2y^2$ and $y^p$ this
     leads to a linear system of equations with extended coefficient matrix
     \begin{displaymath}
       M=\left(
         \begin{array}[m]{ccc|c}
           1&p&0&1\\\lambda&2\lambda&2\lambda&0\\1&0&q&0
         \end{array}
       \right).
     \end{displaymath}
     This system is solvable if and only if the equation
     \begin{displaymath}
       \lambda\cdot \big(pq-2\cdot(p+q)\big)=2\lambda
     \end{displaymath}
     has a solution, i.e.\ if the first $3\times 3$-Minor $-\lambda\cdot\big(pq-2\cdot
     (p+q)\big)\not=0$.
     This shows that
     \begin{displaymath}
       p+q=\tau(f)=\dim_{\K}\big(\gr_P^{AC}(T_f)\big),
     \end{displaymath}
     where for the latter equality we take into account that $\tau(f)$
     is a lower bound for the dimension. Moreover,
     \begin{displaymath}
       B'=\{1,x,\ldots,x^{p-1},y,y^2,\ldots,y^{q-1},xy\}
     \end{displaymath}
     is a regular basis for $T_f$ and $f$ satisfies \AC.

     \emph{Assume next that $\Char(\K)$ does neither divide $p$
       nor $q$, but it divides $pq-2\cdot (p+q)$.}
     We have already seen in the first case that the system of linear
     equations with extended coefficient matrix $M$ is not
     solvable under the given hypotheses. It follows that $x^p$ does
     not lie in $\tj(f)$\tom{, since an equation of the form $x^p=A\cdot
       f+B\cdot f_x+C\cdot f_y$ for some power series necessarily
       implies an equation of the above type}. Therefore, 
     $B$ is a regular basis for $T_f$ and
     \begin{displaymath}
       p+q+1=\tau(f)=\dim_{\K}\big(\gr_P^{AC}(T_f)\big).
     \end{displaymath}
     \tom{For the latter we take into account that
     $\dim_\K\big(\gr_P^A(M_f)\big)=p+q+1$ is an upper bound for
     $\dim_{\K}\big(\gr_P^{AC}(T_f)\big)$ and $\tau(f)$ is a lower bound.}

     \emph{Assume now that $\Char(\K)$ divides $p$ but not $q$.} Then
     it is straight forward to see that
     \begin{displaymath}
       \tj(f)=\big\langle x^p,y^q,xy^2,qy^{q-1}-2\lambda x^2y\big\rangle.
     \end{displaymath}
     and thus $B'$ is a $\K$-vector space basis of $T_f$. We claim
     that $B'$ also generates $\gr_P^{AC}(T_f)$, so that $f$ satisfies
     \AC\  with Tjurina number $\tau(f)=p+q$. 
     \begin{figure}[h]
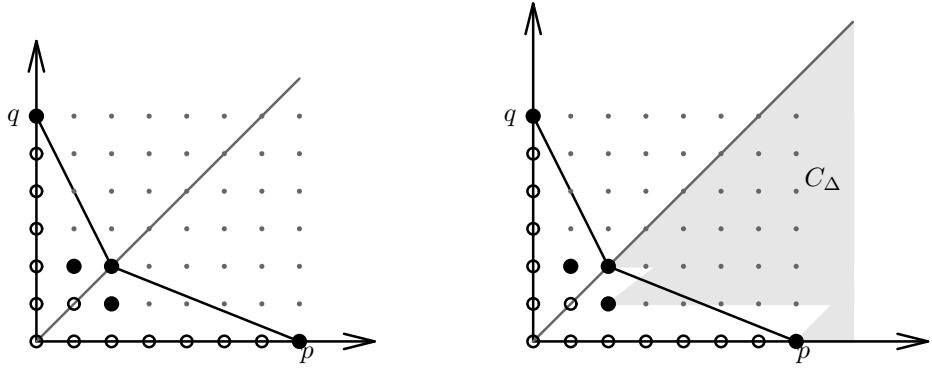

       \centering
       \begin{texdraw}
         \arrowheadtype t:V
         \drawdim cm \relunitscale 0.5 \linewd 0.06 
         \setgray 0
         \move (0 0) \avec (9 0)
         \move (0 0) \avec (0 8)
         \move (0 6) \lvec (2 2) \lvec (7 0)
         \setgray 0.4
         \move (0 0) \lvec (7 7)
         \setgray 0
         \htext (-0.8 5.7){$q$}
         \htext (7 -0.6){$p$}
         \move (0 0) \lcir r:0.15 
         \move (1 0) \lcir r:0.15 
         \move (2 0) \lcir r:0.15 
         \move (3 0) \lcir r:0.15 
         \move (4 0) \lcir r:0.15 
         \move (5 0) \lcir r:0.15 
         \move (6 0) \lcir r:0.15 
         \move (7 0) \fcir f:0 r:0.2 
         \move (0 1) \lcir r:0.15 
         \move (1 1) \lcir r:0.15 
         \move (2 1) \fcir f:0 r:0.2 
         \move (3 1) \fcir f:0.4 r:0.06 
         \move (4 1) \fcir f:0.4 r:0.06 
         \move (5 1) \fcir f:0.4 r:0.06 
         \move (6 1) \fcir f:0.4 r:0.06 
         \move (7 1) \fcir f:0.4 r:0.06 
         \move (0 2) \lcir r:0.15 
         \move (1 2) \fcir f:0 r:0.2 
         \move (2 2) \fcir f:0 r:0.2 
         \move (3 2) \fcir f:0.4 r:0.06 
         \move (4 2) \fcir f:0.4 r:0.06 
         \move (5 2) \fcir f:0.4 r:0.06 
         \move (6 2) \fcir f:0.4 r:0.06 
         \move (7 2) \fcir f:0.4 r:0.06 
         \move (0 3) \lcir r:0.15 
         \move (1 3) \fcir f:0.4 r:0.06 
         \move (2 3) \fcir f:0.4 r:0.06 
         \move (3 3) \fcir f:0.4 r:0.06 
         \move (4 3) \fcir f:0.4 r:0.06 
         \move (5 3) \fcir f:0.4 r:0.06 
         \move (6 3) \fcir f:0.4 r:0.06 
         \move (7 3) \fcir f:0.4 r:0.06 
         \move (0 4) \lcir r:0.15 
         \move (1 4) \fcir f:0.4 r:0.06 
         \move (2 4) \fcir f:0.4 r:0.06 
         \move (3 4) \fcir f:0.4 r:0.06 
         \move (4 4) \fcir f:0.4 r:0.06 
         \move (5 4) \fcir f:0.4 r:0.06 
         \move (6 4) \fcir f:0.4 r:0.06 
         \move (7 4) \fcir f:0.4 r:0.06 
         \move (0 5) \lcir r:0.15 
         \move (1 5) \fcir f:0.4 r:0.06 
         \move (2 5) \fcir f:0.4 r:0.06 
         \move (3 5) \fcir f:0.4 r:0.06 
         \move (4 5) \fcir f:0.4 r:0.06 
         \move (5 5) \fcir f:0.4 r:0.06 
         \move (6 5) \fcir f:0.4 r:0.06          
         \move (7 5) \fcir f:0.4 r:0.06 
         \move (0 6) \fcir f:0 r:0.2 
         \move (1 6) \fcir f:0.4 r:0.06 
         \move (2 6) \fcir f:0.4 r:0.06 
         \move (3 6) \fcir f:0.4 r:0.06 
         \move (4 6) \fcir f:0.4 r:0.06 
         \move (5 6) \fcir f:0.4 r:0.06 
         \move (6 6) \fcir f:0.4 r:0.06          
         \move (7 6) \fcir f:0.4 r:0.06 
       \end{texdraw}
       \hspace{1.5cm}
       \begin{texdraw}
         \arrowheadtype t:V
         \drawdim cm \relunitscale 0.5 \linewd 0.06 
         \setgray 0.9
         \move (2 2) \lvec (8.5 8.5) \lvec (8.5 2) \lvec (2 2) \lfill f:0.9
         \move (2 1) \lvec (8.5 6) \lvec (8.5 1) \lvec (2 1) \lfill f:0.9
         \move (7 0) \lvec (8.5 1.5) \lvec (8.5 0) \lvec (7 0) \lfill f:0.9
         \setgray 0
         \move (0 0) \avec (10.5 0)
         \move (0 0) \avec (0 9)
         \move (0 6) \lvec (2 2) \lvec (7 0)
         \setgray 0.4
         \move (0 0) \lvec (8.5 8.5)
         \setgray 0
         \htext (-0.8 5.7){$q$}
         \htext (7 -0.6){$p$}
         \htext (7.2 4){$C_\Delta$}         
         \move (0 0) \lcir r:0.15 
         \move (1 0) \lcir r:0.15 
         \move (2 0) \lcir r:0.15 
         \move (3 0) \lcir r:0.15 
         \move (4 0) \lcir r:0.15 
         \move (5 0) \lcir r:0.15 
         \move (6 0) \lcir r:0.15 
         \move (7 0) \fcir f:0 r:0.2 
         \move (0 1) \lcir r:0.15 
         \move (1 1) \lcir r:0.15 
         \move (2 1) \fcir f:0 r:0.2 
         \move (3 1) \fcir f:0.4 r:0.06 
         \move (4 1) \fcir f:0.4 r:0.06 
         \move (5 1) \fcir f:0.4 r:0.06 
         \move (6 1) \fcir f:0.4 r:0.06 
         \move (7 1) \fcir f:0.4 r:0.06 
         \move (0 2) \lcir r:0.15 
         \move (1 2) \fcir f:0 r:0.2 
         \move (2 2) \fcir f:0 r:0.2 
         \move (3 2) \fcir f:0.4 r:0.06 
         \move (4 2) \fcir f:0.4 r:0.06 
         \move (5 2) \fcir f:0.4 r:0.06 
         \move (6 2) \fcir f:0.4 r:0.06 
         \move (7 2) \fcir f:0.4 r:0.06 
         \move (0 3) \lcir r:0.15 
         \move (1 3) \fcir f:0.4 r:0.06 
         \move (2 3) \fcir f:0.4 r:0.06 
         \move (3 3) \fcir f:0.4 r:0.06 
         \move (4 3) \fcir f:0.4 r:0.06 
         \move (5 3) \fcir f:0.4 r:0.06 
         \move (6 3) \fcir f:0.4 r:0.06 
         \move (7 3) \fcir f:0.4 r:0.06 
         \move (0 4) \lcir r:0.15 
         \move (1 4) \fcir f:0.4 r:0.06 
         \move (2 4) \fcir f:0.4 r:0.06 
         \move (3 4) \fcir f:0.4 r:0.06 
         \move (4 4) \fcir f:0.4 r:0.06 
         \move (5 4) \fcir f:0.4 r:0.06 
         \move (6 4) \fcir f:0.4 r:0.06 
         \move (7 4) \fcir f:0.4 r:0.06 
         \move (0 5) \lcir r:0.15 
         \move (1 5) \fcir f:0.4 r:0.06 
         \move (2 5) \fcir f:0.4 r:0.06 
         \move (3 5) \fcir f:0.4 r:0.06 
         \move (4 5) \fcir f:0.4 r:0.06 
         \move (5 5) \fcir f:0.4 r:0.06 
         \move (6 5) \fcir f:0.4 r:0.06          
         \move (7 5) \fcir f:0.4 r:0.06 
         \move (0 6) \fcir f:0 r:0.2 
         \move (1 6) \fcir f:0.4 r:0.06 
         \move (2 6) \fcir f:0.4 r:0.06 
         \move (3 6) \fcir f:0.4 r:0.06 
         \move (4 6) \fcir f:0.4 r:0.06 
         \move (5 6) \fcir f:0.4 r:0.06 
         \move (6 6) \fcir f:0.4 r:0.06          
         \move (7 6) \fcir f:0.4 r:0.06 
       \end{texdraw}

       \caption{On the left hand side the elements of $B'$ are
         depicted by large white dots and the elements in $B^c$ are
         depicted by large black dots. The right hand side shows the
         union of $(2,2)+C_{\Delta}$, $(2,1)+C_\Delta$ and
         $(p,0)+C_\Delta$ which covers all lattice points in
         $C_\Delta$ which are not in $B'$.}
       \label{fig:tpqcones} 
     \end{figure}
     By Lemma~\ref{lem:cone} it suffices to check that the monomials
     in
     \begin{displaymath}
       B^c=\big\{x^2y^2,xy^2,x^2y,x^p,y^q\}
     \end{displaymath}
     are zero in $\gr_P^{AC}(T_f)$ (see Figure~\ref{fig:tpqcones}). 
     However, we have that
     \begin{displaymath}
       x^2y^2=\frac{1}{2\lambda}\cdot x\cdot \partial_x f \in\tj(f)_{2pq}
     \end{displaymath}
     and
     \begin{displaymath}
       x^p=f-\left(\frac{1}{2}+\frac{1}{q}\right)\cdot
       x\cdot\partial_x f-\frac{1}{q}\cdot y\cdot\partial_y f
       \in\tj(f)_{2pq}.
     \end{displaymath}
     Moreover, $v_P(xy^2)=pq+2q$ and $v_P(\partial_x)=2p-pq$ so that
     \begin{displaymath}
       xy^2=\frac{1}{2\lambda}\cdot \partial_x f\in \tj(f)_{pq+2p}.
     \end{displaymath}
     Similar arguments hold for $x^2y$ and $y^q$. This finishes this
     case.

     \emph{Assume now that $\Char(\K)$ divides $q$ but not $p$.}
     This case follows by symmetry from the previous case, i.e.\ $f$
     is \AC\  with Tjurina number $\tau(f)=p+q$.

     \emph{Assume finally that $\Char(\K)$ divides both $p$ and $q$.} 
     Then $\tj(f)=\langle xy^2,x^2y,x^p+y^q,x^{p+1},y^{q+1}\rangle$
     and $B$ is a $\K$-vector space basis of $T_f$. Moreover, we claim that it
     is a regular basis for $T_f$ as well. By Lemma~\ref{lem:cone} it
     suffices to show that the monomials $x^{p+1}$, $y^{q+1}$,
     $x^2y^2$, $xy^2$ and $x^2y$ as well as the binomial $x^p+y^q$ are
     zero in $\gr_P^{AC}(T_f)$. This can be achieved in the same way
     as above. In particular we have
     \begin{displaymath}
       p+q+1=\tau(f)=\dim_\K\big(\gr_P^{AC}(T_f)\big)
     \end{displaymath}
     and $f$ satisfies \AC.

     \emph{Conclusion:} In each of the above cases the regular basis
     $B$ respectively $B'$ consists of monomials on or below the
     Newton diagram $P=\Gamma(f)$. Therefore, the normal form
     algorithm shows that any power series with principal part $f$
     with respect to $P$ has indeed $f$ as normal form.\hfill$\Box$
   \end{example}

   \begin{corollary}[Normal form of $T_{pq}$-Singularities]\label{cor:tpq}
     Suppose that $\Char(\K)\not=2$ and let $f\in\K[[x,y]]$ be a power
     series with $\In_{\Gamma(f)}=x^p+\lambda\cdot x^2y^2+y^q$,
     $\frac{1}{p}+\frac{1}{q}<\frac{1}{2}$ and $\lambda\not=0$. Then $f$ is \AC\  and 
     \begin{displaymath}
       f\;\sim_c\;x^p+\lambda\cdot x^2y^2+y^q.
     \end{displaymath}
     The contact determinacy of $f$ is $\max\{p,q\}$.
   \end{corollary}
   \begin{proof}
     That $f$ satisfies \AC\  and is contact equivalent to its
     principal part was shown in Example~\ref{ex:tpq}. It is obvious
     that any monomial above the Newton diagram has stricly larger
     piecewise valuation than $f$. Using the
     notation from Example~\ref{ex:tpq} it follows that
     $\m^{k+1}\subseteq F_{2pq+1}$ for $k=\max\{p,q\}$, so that by
     Corollary~\ref{cor:fdce} the degree of contact determinacy is at
     most $k$. To see that it cannot be less we may assume the
     contrary and we may assume moreover that $p\geq q$. Then $f\sim_c
     \In_{\Gamma(f)}(f)-x^p=\lambda\cdot x^2y^2+y^q$, but the latter is
     non-reduced and has thus infinite Tjurina number. This is clearly
     a contradiction.
   \end{proof}

   \begin{remark}
     If $\Char(\K)=2$ neither the conclusion in
     Corollary~\ref{cor:tpq} nor the investigation in
     Example~\ref{ex:tpq} hold 
     in general as we can see from Example~\ref{ex:t45}. 
   \end{remark}

   For the $T_{pq}$-Singularities we considered the conditions \Aa\  and
   \AC\  and deduced a normal form. However, for normal forms we only
   need a good way to choose a 
   \emph{small} regular basis for $M_{\In_P(f)}$ respectively $T_{\In_P(f)}$. There the
   following observations are useful.

   \begin{remark}\label{rem:aaaac}
     Each $C$-polytope $P$ has only finitely many zero-dimensional
     faces and each facet is the convex hull of some of these. The cones over
     these zero-dimensional faces are rays, and for each facet
     $\Delta$ of $P$ the cone $C_\Delta$ is spanned by a finite number
     of these rays, none of which is superfluous, i.e.\ they are the
     extremal rays of the cone.

     \emph{Then $f\in\Kx$ satisfies condition \AAa\  respectively
       \AAC\ w.r.t.\ $P$  if and
     only if on each ray spanned by a zero-dimensional face of $P$
     there is a lattice point $\alpha$ such 
     that $\ux^\alpha$ is zero in $\gr_P^A(M_f)$ respectively in
     $\gr_P^{AC}(T_f)$.}
   \end{remark}
   \begin{proof}
     Consider first the two-dimensional situation such that each cone $C_\Delta$ is
     spanned by two rays. Suppose that a cone $C_\Delta$ is given and it
     is spanned by the rays $r$ and $s$, and suppose that $\alpha$ is
     a lattice point on $r$ and $\beta$ is a lattice point on $s$.
     \begin{figure}[h]
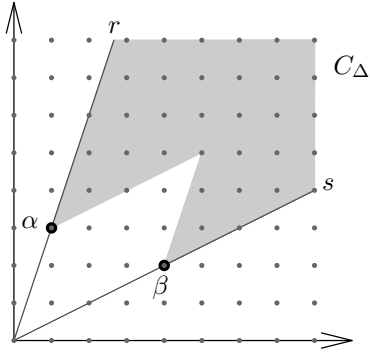

       \centering
       \begin{texdraw}
         \arrowheadtype t:V
         \drawdim cm \relunitscale 0.5 \linewd 0.04 
         \move (0 0) \avec (9 0) \move (0 0) \avec (0 9)
         \setgray 0.8
         \move (4 2) \lvec (8 4) \lvec (8 8) \lvec (6 8) 
         \lvec (4 2) \lfill f:0.8         
         \move (1 3) \lvec (2.66 8) \lvec (8 8) \lvec (8 6.5) 
         \lvec (1 3) \lfill f:0.8
         \setgray 0.3
         \move (0 0) \lvec (8 4)
         \move (0 0) \lvec (2.66 8)
         \move (1 3) \fcir f:0 r:0.15
         \move (4 2) \fcir f:0 r:0.15
         \htext (2.5 8.2){$r$}
         \htext (8.2 4){$s$}
         \htext (0.2 3){$\alpha$}
         \htext (3.7 1.1){$\beta$}
         \htext (8.5 7){$C_\Delta$}
         \move (0 0) \fcir f:0.4 r:0.06 
         \move (1 0) \fcir f:0.4 r:0.06 
         \move (2 0) \fcir f:0.4 r:0.06 
         \move (3 0) \fcir f:0.4 r:0.06 
         \move (4 0) \fcir f:0.4 r:0.06 
         \move (5 0) \fcir f:0.4 r:0.06 
         \move (6 0) \fcir f:0.4 r:0.06 
         \move (7 0) \fcir f:0.4 r:0.06 
         \move (8 0) \fcir f:0.4 r:0.06 
         \move (0 1) \fcir f:0.4 r:0.06 
         \move (1 1) \fcir f:0.4 r:0.06 
         \move (2 1) \fcir f:0.4 r:0.06 
         \move (3 1) \fcir f:0.4 r:0.06 
         \move (4 1) \fcir f:0.4 r:0.06 
         \move (5 1) \fcir f:0.4 r:0.06 
         \move (6 1) \fcir f:0.4 r:0.06 
         \move (7 1) \fcir f:0.4 r:0.06 
         \move (8 1) \fcir f:0.4 r:0.06 
         \move (0 2) \fcir f:0.4 r:0.06 
         \move (1 2) \fcir f:0.4 r:0.06 
         \move (2 2) \fcir f:0.4 r:0.06 
         \move (3 2) \fcir f:0.4 r:0.06 
         \move (4 2) \fcir f:0.4 r:0.06 
         \move (5 2) \fcir f:0.4 r:0.06 
         \move (6 2) \fcir f:0.4 r:0.06 
         \move (7 2) \fcir f:0.4 r:0.06 
         \move (8 2) \fcir f:0.4 r:0.06 
         \move (0 3) \fcir f:0.4 r:0.06 
         \move (1 3) \fcir f:0.4 r:0.06 
         \move (2 3) \fcir f:0.4 r:0.06 
         \move (3 3) \fcir f:0.4 r:0.06 
         \move (4 3) \fcir f:0.4 r:0.06 
         \move (5 3) \fcir f:0.4 r:0.06 
         \move (6 3) \fcir f:0.4 r:0.06 
         \move (7 3) \fcir f:0.4 r:0.06 
         \move (8 3) \fcir f:0.4 r:0.06 
         \move (0 4) \fcir f:0.4 r:0.06 
         \move (1 4) \fcir f:0.4 r:0.06 
         \move (2 4) \fcir f:0.4 r:0.06 
         \move (3 4) \fcir f:0.4 r:0.06 
         \move (4 4) \fcir f:0.4 r:0.06 
         \move (5 4) \fcir f:0.4 r:0.06 
         \move (6 4) \fcir f:0.4 r:0.06 
         \move (7 4) \fcir f:0.4 r:0.06 
         \move (8 4) \fcir f:0.4 r:0.06 
         \move (0 5) \fcir f:0.4 r:0.06 
         \move (1 5) \fcir f:0.4 r:0.06 
         \move (2 5) \fcir f:0.4 r:0.06 
         \move (3 5) \fcir f:0.4 r:0.06 
         \move (4 5) \fcir f:0.4 r:0.06 
         \move (5 5) \fcir f:0.4 r:0.06 
         \move (6 5) \fcir f:0.4 r:0.06                  
         \move (7 5) \fcir f:0.4 r:0.06 
         \move (8 5) \fcir f:0.4 r:0.06 
         \move (0 6) \fcir f:0.4 r:0.06 
         \move (1 6) \fcir f:0.4 r:0.06 
         \move (2 6) \fcir f:0.4 r:0.06 
         \move (3 6) \fcir f:0.4 r:0.06 
         \move (4 6) \fcir f:0.4 r:0.06 
         \move (5 6) \fcir f:0.4 r:0.06 
         \move (6 6) \fcir f:0.4 r:0.06                  
         \move (7 6) \fcir f:0.4 r:0.06 
         \move (8 6) \fcir f:0.4 r:0.06 
         \move (0 7) \fcir f:0.4 r:0.06 
         \move (1 7) \fcir f:0.4 r:0.06 
         \move (2 7) \fcir f:0.4 r:0.06 
         \move (3 7) \fcir f:0.4 r:0.06 
         \move (4 7) \fcir f:0.4 r:0.06 
         \move (5 7) \fcir f:0.4 r:0.06 
         \move (6 7) \fcir f:0.4 r:0.06                  
         \move (7 7) \fcir f:0.4 r:0.06 
         \move (8 7) \fcir f:0.4 r:0.06 
         \move (0 8) \fcir f:0.4 r:0.06 
         \move (1 8) \fcir f:0.4 r:0.06 
         \move (2 8) \fcir f:0.4 r:0.06 
         \move (3 8) \fcir f:0.4 r:0.06 
         \move (4 8) \fcir f:0.4 r:0.06 
         \move (5 8) \fcir f:0.4 r:0.06 
         \move (6 8) \fcir f:0.4 r:0.06                  
         \move (7 8) \fcir f:0.4 r:0.06 
         \move (8 8) \fcir f:0.4 r:0.06 
       \end{texdraw}
       \caption{$C_\Delta$ almost filled by two shifted copies of $C_\Delta$.}
       \label{fig:shiftedcones}
     \end{figure}
     The shifted rays $\alpha+s$ and $\beta+r$ will intersect, since
     $r$ and $s$ are not parallel, and thus the rays $r$, $s$,
     $\alpha+s$ and $\beta+r$ bound a finite region in the cone
     $C_\Delta$ (see Figure~\ref{fig:shiftedcones}). By
     Lemma~\ref{lem:cone} the lattice points which are not inside the 
     bounded region will be zero in $\gr_P^A(M_f)$ respectively in
     $\gr_P^{AC}(T_f)$. We can play this game for each facet of $P$,
     and thus there are only finitely many monomials whose class is
     not zero.

     The argument generalises right away to higher dimensions.
   \end{proof}

   \begin{corollary}\label{cor:abcd}
     Let $f=x^a+y^b+\lambda\cdot x^cy^d\in\K[[x,y]]$ with
     $\lambda\in\K$, $a>c\geq 1$, $b>d\geq 1$ and $ad+bc<ab$, and let
     $P=\Gamma(f)$. 
     Then $f$ satisfies \AAa\  respectively \AAC\ w.r.t.\ $P$ if and only if there
     are natural numbers $k,m,n$ such that $x^m$, $y^n$ and
     $x^{ck}y^{dk}$ are zero in $\gr_P^A(M_f)$ respectively in
     $\gr_P^{AC}(T_f)$. 
   \end{corollary}
   \begin{proof}
     The Newton diagram of $f$ is schematically shown in
     Figure~\ref{fig:abcd}, and the result follows from
     Remark~\ref{rem:aaaac}.  
     \begin{figure}[h]
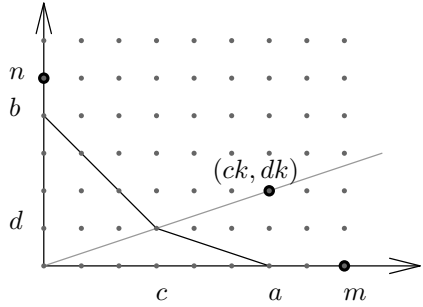

       \centering
       \begin{texdraw}
         \arrowheadtype t:V
         \drawdim cm \relunitscale 0.5 \linewd 0.04 
         \move (0 0) \avec (10 0) \move (0 0) \avec (0 7)
         \setgray 0.6
         \move (0 0) \lvec (9 3)
         \setgray 0
         \move (6 0) \lvec (3 1) \lvec (0 4)
         \move (0 5) \fcir f:0 r:0.15
         \move (8 0) \fcir f:0 r:0.15
         \move (6 2) \fcir f:0 r:0.15
         \htext (8 -0.9){$m$}
         \htext (6 -0.9){$a$}
         \htext (3 -0.9){$c$}
         \htext (-0.9 5){$n$}
         \htext (-0.9 4){$b$}
         \htext (-0.9 1){$d$}
         \htext (4.5 2.2){$(ck,dk)$}
         \move (0 0) \fcir f:0.4 r:0.06 
         \move (1 0) \fcir f:0.4 r:0.06 
         \move (2 0) \fcir f:0.4 r:0.06 
         \move (3 0) \fcir f:0.4 r:0.06 
         \move (4 0) \fcir f:0.4 r:0.06 
         \move (5 0) \fcir f:0.4 r:0.06 
         \move (6 0) \fcir f:0.4 r:0.06 
         \move (7 0) \fcir f:0.4 r:0.06 
         \move (8 0) \fcir f:0.4 r:0.06 
         \move (0 1) \fcir f:0.4 r:0.06 
         \move (1 1) \fcir f:0.4 r:0.06 
         \move (2 1) \fcir f:0.4 r:0.06 
         \move (3 1) \fcir f:0.4 r:0.06 
         \move (4 1) \fcir f:0.4 r:0.06 
         \move (5 1) \fcir f:0.4 r:0.06 
         \move (6 1) \fcir f:0.4 r:0.06 
         \move (7 1) \fcir f:0.4 r:0.06 
         \move (8 1) \fcir f:0.4 r:0.06 
         \move (0 2) \fcir f:0.4 r:0.06 
         \move (1 2) \fcir f:0.4 r:0.06 
         \move (2 2) \fcir f:0.4 r:0.06 
         \move (3 2) \fcir f:0.4 r:0.06 
         \move (4 2) \fcir f:0.4 r:0.06 
         \move (5 2) \fcir f:0.4 r:0.06 
         \move (6 2) \fcir f:0.4 r:0.06 
         \move (7 2) \fcir f:0.4 r:0.06 
         \move (8 2) \fcir f:0.4 r:0.06 
         \move (0 3) \fcir f:0.4 r:0.06 
         \move (1 3) \fcir f:0.4 r:0.06 
         \move (2 3) \fcir f:0.4 r:0.06 
         \move (3 3) \fcir f:0.4 r:0.06 
         \move (4 3) \fcir f:0.4 r:0.06 
         \move (5 3) \fcir f:0.4 r:0.06 
         \move (6 3) \fcir f:0.4 r:0.06 
         \move (7 3) \fcir f:0.4 r:0.06 
         \move (8 3) \fcir f:0.4 r:0.06 
         \move (0 4) \fcir f:0.4 r:0.06 
         \move (1 4) \fcir f:0.4 r:0.06 
         \move (2 4) \fcir f:0.4 r:0.06 
         \move (3 4) \fcir f:0.4 r:0.06 
         \move (4 4) \fcir f:0.4 r:0.06 
         \move (5 4) \fcir f:0.4 r:0.06 
         \move (6 4) \fcir f:0.4 r:0.06 
         \move (7 4) \fcir f:0.4 r:0.06 
         \move (8 4) \fcir f:0.4 r:0.06 
         \move (0 5) \fcir f:0.4 r:0.06 
         \move (1 5) \fcir f:0.4 r:0.06 
         \move (2 5) \fcir f:0.4 r:0.06 
         \move (3 5) \fcir f:0.4 r:0.06 
         \move (4 5) \fcir f:0.4 r:0.06 
         \move (5 5) \fcir f:0.4 r:0.06 
         \move (6 5) \fcir f:0.4 r:0.06                  
         \move (7 5) \fcir f:0.4 r:0.06 
         \move (8 5) \fcir f:0.4 r:0.06 
         \move (0 6) \fcir f:0.4 r:0.06 
         \move (1 6) \fcir f:0.4 r:0.06 
         \move (2 6) \fcir f:0.4 r:0.06 
         \move (3 6) \fcir f:0.4 r:0.06 
         \move (4 6) \fcir f:0.4 r:0.06 
         \move (5 6) \fcir f:0.4 r:0.06 
         \move (6 6) \fcir f:0.4 r:0.06                  
         \move (7 6) \fcir f:0.4 r:0.06 
         \move (8 6) \fcir f:0.4 r:0.06 
       \end{texdraw}
       \caption{The Newton diagram of $x^a+y^b+\lambda\cdot x^cy^d$.}
       \label{fig:abcd}
     \end{figure}
   \end{proof}

   \begin{example}[$E_{3,3}$-Singularity in characteristic $3$]\label{ex:e33}
     Let $\Char(\K)=3$ and consider the equation
     \begin{displaymath}
       f=x^{12}+x^3y^2+y^3\in\K[[x,y]].
     \end{displaymath}
     $f$ is piecewise homogeneous with respect to its Newton diagram
     and using the procedure \texttt{isAC} from the \textsc{Singular}
     library \texttt{gradalg.lib} we can check that $x^{15}$, $y^{15}$
     and $x^9y^6$ are zero in $\gr_P^{AC}(T_f)$. Thus $f$ is \AAC\ 
     with respect to $\Gamma(f)$ by
     Corollary~\ref{cor:abcd}. Moreover, using the procedure
     \texttt{ACgrbase} 
     from the same library we can compute the regular basis
     \begin{displaymath}
       B=\{1,x,\ldots,x^{12},y,xy,x^2y,y^2,xy^2,x^2y^2,xy^3,x^2y^3,x^2y^4\}
     \end{displaymath}
     for $T_f$. Hence, $\dim_\K\big(\gr_P^{AC}(T_f)\big)=|B|=22$ while
     $\tau(f)=21$. This shows that $f$ is \emph{not} \AC. 

     Theorem~\ref{thm:nfce} shows that any power series $g$ whose
     principal part with respect to $\Gamma(f)$ is $f$ satisfies
     \begin{displaymath}
       g\;\sim_c\;f+c_1\cdot xy^3+c_2\cdot x^2y^3+c_3\cdot x^2y^4
     \end{displaymath}
     for suitable $c_1,c_2,c_3\in\K$. 

     Scaling the weight vectors corresponding to the facets of
     $\Gamma(f)$ suitably they are $w_1=(6,27)$ and $w_2=(8,24)$, and
     $f$ is \ph\  of piecewise degree $72$. Moreover, the maximum of the
     piecewise degree of the monomials in $B$ is $d=112$, and an easy
     computation shows that $\m^{19}\subseteq F_{113}$. By
     Corollary~\ref{cor:fdce} we therefore know that the contact
     determinacy bound 
     of $f$ is at most $18$. That is much better than the bound
     \begin{displaymath}
       2\cdot\tau(f)-\ord(f)+2=41
     \end{displaymath}
     which \cite[Theorem~2.1]{BGM10} gives.\hfill$\Box$
   \end{example}

   \begin{example}[$E_7$-Singularity]
     Let $f\in\K[[x,y,z]]$ and let $P$ be a $C$-polytope 
     containing $\Gamma(f)$ and suppose that the  principal part of
     $f$ is  $\In_P(f)=x^3+xy^3+z^2$. Then $f$ is \csqh\  and our
     methods show the following (for the 
     details we refer to \cite[Example~3.3.22]{Bou09}):

     \emph{1st Case: $\Char(\K)\not\in\{2,3\}$}: Then
     $f\;\sim_c\;\In_P(f)$, $\tau(f)=\tau\big(\In_P(f)\big)=7$ and the
     contact determinacy is $4$.  

     \emph{2nd Case: $\Char(\K)=3$}: Then $f\;\sim_c\;\In_P(f)+c\cdot
     x^2y^2$ for some $c\in\K$, $\tau\big(\In_P(f)\big)=9$ and the
     contact determinacy is again $4$. If $c\not=0$ then
     $\tau(f)=7$. 

     \emph{3rd Case: $\Char(\K)=2$}: Then
     $f\;\sim_c\;\In_P(f)+c_1\cdot y^3z+c_4\cdot y^4z$,
     $\tau\big(\In_P(f)\big)=14$ and the determinacy is
     $5$.      
   \end{example}

   \begin{example}[$W_{1,1}$-Singularities] \label{ex:w11}
     Let $f\in\K[[x,y]]$ be such that the principal part with respect
     to $P=\Gamma(f)$ is  $\In_P(f)=x^7+x^3y^2+y^4$. Then $f$ is \csph\ 
     and our methods give the 
     following normal forms (for the details we refer to
     \cite[Example~3.3.9]{Bou09}):

     \emph{1st Case: $\Char(\K)\not\in\{2,3\}$}: Then
     $f\;\sim_c\;\In_P(f)$. 

     \emph{2nd Case: $\Char(\K)=3$}: Then $f\;\sim_c\;\In_P(f)+
     c_1\cdot xy^4+c_2\cdot x^2y^3+ c_3\cdot x^2y^4 +c_4\cdot
     x^2y^5$ for some $c_1,\ldots,c_4\in\K$. However, considering
     parametrisations it can be shown 
     that actually $f\;\sim_c\;\In_P(f)+c\cdot x^2y^3$ for some
     $c\in\K$ (see \cite{Bou02}).

     \emph{3rd Case: $\Char(\K)=2$}: Then $f\;\sim_c\;\In_P(f)+c\cdot
     x^6y$ for some $c\in\K$.

     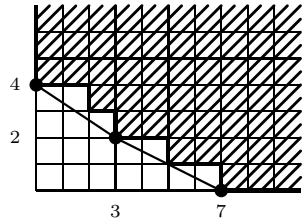
\begin{figure}[h]
       \centering
       \begin{picture}(100,80)
         \put(10,10){\line(1,0){100}}
         \put(10,20){\line(1,0){100}}
         \put(10,30){\line(1,0){100}}
         \put(10,40){\line(1,0){100}}
         \put(10,50){\line(1,0){100}}
         \put(10,60){\line(1,0){100}}
         \put(10,70){\line(1,0){100}}         
         \put(10,10){\line(0,1){70}}
         \put(20,10){\line(0,1){70}}
         \put(30,10){\line(0,1){70}}
         \put(40,10){\line(0,1){70}}
         \put(50,10){\line(0,1){70}}
         \put(60,10){\line(0,1){70}}
         \put(70,10){\line(0,1){70}}
         \put(80,10){\line(0,1){70}}
         \put(90,10){\line(0,1){70}}
         \put(100,10){\line(0,1){70}}
         \thicklines\drawline[12](10,50)(40,30)
         \thicklines\drawline[12](40,30)(80,10)         
         \linethickness{0.2mm}\scriptsize
         \Thicklines
         \put(10,50){\line(0,1){30}}
         \put(10,50){\line(1,0){20}}
         \put(30,40){\line(0,1){10}}
         \put(30,40){\line(1,0){10}}
         \put(40,30){\line(0,1){10}}
         \put(40,30){\line(1,0){20}}
         \put(60,20){\line(0,1){10}}
         \put(60,20){\line(1,0){20}}
         \put(80,10){\line(0,1){10}}
         \put(80,10){\line(1,0){30}}
         \put(10,75){\line(1,1){5}}
         \put(10,70){\line(1,1){10}}
         \put(10,65){\line(1,1){15}}
         \put(10,60){\line(1,1){20}}
         \put(10,55){\line(1,1){25}}
         \put(10,50){\line(1,1){30}}
         \put(15,50){\line(1,1){30}}
         \put(20,50){\line(1,1){30}}
         \put(25,50){\line(1,1){30}}
         \put(30,50){\line(1,1){30}}
         \put(30,45){\line(1,1){35}}
         \put(30,40){\line(1,1){40}}
         \put(35,40){\line(1,1){40}}
         \put(40,40){\line(1,1){40}}
         \put(40,35){\line(1,1){45}}
         \put(40,30){\line(1,1){50}}
         \put(45,30){\line(1,1){50}}
         \put(50,30){\line(1,1){50}}
         \put(55,30){\line(1,1){50}}
         \put(60,30){\line(1,1){50}}
         \put(60,25){\line(1,1){50}}
         \put(60,20){\line(1,1){50}}
         \put(65,20){\line(1,1){45}}
         \put(70,20){\line(1,1){40}}
         \put(75,20){\line(1,1){35}}
         \put(80,20){\line(1,1){30}}
         \put(80,15){\line(1,1){30}}
         \put(80,10){\line(1,1){30}}
         \put(85,10){\line(1,1){25}}
         \put(90,10){\line(1,1){20}}
         \put(95,10){\line(1,1){15}}
         \put(100,10){\line(1,1){10}}
         \put(105,10){\line(1,1){5}}
         \put(10,50){\circle*{4}}
         \put(40,30){\circle*{4}}
         \put(80,10){\circle*{4}}
         \put(78,0){$7$}
         \put(38,0){$3$}
         \put(0,28){$2$}
         \put(0,48){$4$}
         \put(5,-15)

       \end{picture}
       \caption{The Newton diagram of $x^7+x^3y^2+y^4$ for
         $\Char(\K)\not=2,3,7$.} 
       \label{fig:npw11}
     \end{figure}
     \hfill$\Box$
   \end{example}

   \nocite{AGV85}
   \nocite{AGV88}


\providecommand{\bysame}{\leavevmode\hbox to3em{\hrulefill}\thinspace}

\end{document}